\documentclass[a4paper,16pt]{article}
\usepackage[utf8]{inputenc}
\usepackage{amssymb,amsmath,mathrsfs}
\usepackage{amsthm}
\usepackage[colorlinks,
            linkcolor=blue,
            anchorcolor=red,
            citecolor=green
            ]{hyperref}
\numberwithin{equation}{section}
\usepackage{graphicx}
\usepackage{amsmath}
\newtheorem{theorem}{Theorem}[section]

\theoremstyle{plain}
\newtheorem{thm}[theorem]{Theorem } 
\newtheorem{defi}[theorem]{Definition }
\newtheorem{prop}[theorem]{Proposition }
\newtheorem{cor}[theorem]{Corollary }

\newtheorem{ques}[theorem]{Question}
\newtheorem{lem}[theorem]{Lemma }

\newtheorem{rem}[theorem] {Remark}

\usepackage{tikz}
\usetikzlibrary{matrix}

\begin{document}

\title{The matrix  Szeg\H{o} equation}
\author{Ruoci Sun\footnote{School of Mathematics, Georgia Institute of Technology, Atlanta, USA. Email: ruoci.sun.16@normalesup.org}}

\maketitle

\noindent $\mathbf{Abstract}$ \quad This paper is dedicated to studying the matrix solutions to the cubic Szeg\H{o} equation, introduced in G\'erard--Grellier \cite{GGANNENS},  leading to the cubic matrix Szeg\H{o} equation on the torus,
\begin{equation*}
 i \partial_t U = \Pi_{\geq 0} \left(U  U ^* U \right), \quad \Pi_{\geq 0}: \sum_{n\in \mathbb{Z}}\hat{U}(n) e^{inx} \mapsto \sum_{n \geq 0} \hat{U}(n) e^{inx}.
 \end{equation*}This equation enjoys a two-Lax-pair structure, which allow every solution to be expressed explicitly in terms of the initial data $U(0)$ and the time $t \in \mathbb{R}$.\\

\noindent $\mathbf{Keywords}$ \quad Szeg\H{o} operator, Lax pair, explicit formula, Hankel operators, Toeplitz operators.\\

\tableofcontents
\section{Introduction}
\noindent Given two arbitrary positive integers $M,N \in \mathbb{N}_+$, the cubic $M \times N$ matrix Szeg\H{o} equation on the torus $\mathbb{T}= \mathbb{R}\slash 2\pi \mathbb{Z}$ reads as
\begin{equation}\label{MSzego}
 i \partial_t U = \Pi_{\geq 0} \left(U  U ^* U \right), \quad U=U(t,x) \in \mathbb{C}^{M \times N}, \quad (t,x)\in \mathbb{R} \times \mathbb{T},
 \end{equation}where $\Pi_{\geq 0}: L^2(\mathbb{T}; \mathbb{C}^{M \times N}) \to L^2(\mathbb{T}; \mathbb{C}^{M \times N})$ denotes Szeg\H{o} operator which cancels all negative Fourier modes and preserves the nonnegative Fourier modes, i.e.
\begin{equation}\label{MSzegoop}
 \left(\Pi_{\geq 0}  U\right)(x) = \sum_{n\geq 0}\hat{U}(n) e^{inx}, \quad \hat{U}(n) = \tfrac{1}{2\pi}\int_{0}^{2\pi} U(x)e^{-inx} \mathrm{d}x \in \mathbb{C}^{M \times N},  \quad \forall n \in \mathbb{Z},
\end{equation}for any $U  =\sum_{n\in \mathbb{Z}}\hat{U}(n) e^{inx}\in  L^2(\mathbb{T}; \mathbb{C}^{M \times N})$.
 
\subsection{Motivation}
\noindent The motivation to introduce equation \eqref{MSzego} is based on the following two facts. On the one hand, the cubic scalar Szeg\H{o} equation on the torus $\mathbb{T}$,
\begin{equation}\label{sSzego}
i\partial_t u = \Pi_{\geq 0}(|u|^2 u), \quad u=u(t,x) \in \mathbb{C}, \quad (t,x) \in \mathbb{R} \times \mathbb{T}, \quad \Pi_{\geq 0}: \sum_{n\in \mathbb{Z}}a_n e^{inx} \mapsto \sum_{n \geq 0}a_n e^{inx}, 
\end{equation}which is introduced in G\'erard--Grellier \cite{GGANNENS, GGinvent, GGAPDE,  GGTurb2015, GGTAMS, GerardGrellierBook} and  G\'erard--Pushnitski \cite{GerPush2023}, is a model of a nondispersive Hamiltonian equation. It enjoys a two-Lax-pair structure, which allows to establish action--angle coordinates on the finite rank manifolds and the explicit formula for an arbitrary $L^2$-solution, leading to the complete integrability of \eqref{sSzego}.  Thanks to its integrable system structure, P. G\'erard and S. Grellier also construct  weakly turbulent solutions, obtain  the quasi-periodicity of rational solutions and the classification of stationary and traveling waves. The cubic scalar Szeg\H{o} equation on the line $\mathbb{R}$ is studied in the works Pocovnicu \cite{poAPDE2011, poDCDS2011} and  G\'erard--Pushnitski \cite{GerPush2023}.\\

\noindent  On the other hand, if we consider the matrix generalization of the following integrable systems, the corresponding matrix equation still enjoys the Lax pair structure. Given any $d \in \mathbb{N}_+$, the filtered Sobolev spaces are given by $H^s_+(\mathbb{T}; \mathbb{C}^{N \times d}):=\Pi_{\geq 0}(H^s(\mathbb{T}; \mathbb{C}^{N \times d}))$, $\forall s \geq 0$. The right Toeplitz operator of symbol $V \in L^2(\mathbb{T}; \mathbb{C}^{M \times N})$ is defined by 
\begin{equation}\label{Rtoep}
\begin{split}
& \mathbf{T}^{(\mathbf{r})}_{V} : G \in H^1_+(\mathbb{T}; \mathbb{C}^{N \times d}) \mapsto \mathbf{T}^{(\mathbf{r})}_{V}(G) = \Pi_{\geq 0}(V G) \in L^2_+(\mathbb{T}; \mathbb{C}^{M \times d}).\\
\end {split}
\end{equation}
\begin{enumerate}
    \item The matrix cubic intertwined derivative Schr\"odinger system of Calogero--Moser--Sutherland type on $\mathbb{T}$, which is introduced in Sun \cite{SUNInterCMSDNLS}
\begin{equation}\label{IdnlsCMS}
\begin{cases}
\partial_t U = i \partial_x^2 U +   U  \partial_x \Pi_{\geq 0}\left(V^* U \right) +  V \partial_x \Pi_{\geq 0}\left(U^* U \right),\\
 \partial_t V = i \partial_x^2 V  +  V \partial_x \Pi_{\geq 0}\left(U^* V \right) +   U  \partial_x \Pi_{\geq 0}\left(V^* V \right),
\end {cases}  \quad  (t,x) \in \mathbb{R} \times \mathbb{T}, 
\end{equation}where $U=U(t), V=V(t) \in H^2_+ (\mathbb{T}; \mathbb{C}^{M \times N})$, enjoys the following Lax pair structure 
\begin{equation*} 
\begin{split}
\mathbf{L}_{(U,V)}^{\mathrm{dNLS}}  = &  \mathrm{D} - \tfrac{1}{2}\left(\mathbf{T}_{U }^{(\mathbf{r})} \mathbf{T}_{V^*}^{(\mathbf{r})} + \mathbf{T}_{V}^{(\mathbf{r})} \mathbf{T}_{U^*}^{(\mathbf{r})}\right), \quad \mathrm{D}=-i \partial_x, \quad \forall U, V \in H^1_+ (\mathbb{T}; \mathbb{C}^{M \times N}), \\ 
\mathbf{B}_{(U,V)}^{\mathrm{dNLS}}   = & \tfrac{1}{2}\left(\mathbf{T}_{U}^{(\mathbf{r})} \mathbf{T}_{\partial_x V^*}^{(\mathbf{r})}  + \mathbf{T}_{V}^{(\mathbf{r})} \mathbf{T}_{\partial_x U^*}^{(\mathbf{r})}   - \mathbf{T}_{\partial_x V}^{(\mathbf{r})} \mathbf{T}_{U^*}^{(\mathbf{r})}  - \mathbf{T}_{\partial_x U}^{(\mathbf{r})} \mathbf{T}_{V^*}^{(\mathbf{r})} \right) +  \tfrac{i}{4}\left(\mathbf{T}_{U}^{(\mathbf{r})} \mathbf{T}_{V^*}^{(\mathbf{r})}  + \mathbf{T}_{V}^{(\mathbf{r})} \mathbf{T}_{U^*}^{(\mathbf{r})}  \right)^2.  
\end{split}
\end{equation*}where $\mathbf{L}_{(U,V)}^{\mathrm{dNLS}},\mathbf{B}_{(U,V)}^{\mathrm{dNLS}}: H^1_+(\mathbb{T}; \mathbb{C}^{M \times d}) \to  L^2_+(\mathbb{T}; \mathbb{C}^{M \times d})$ are densely defined operators. The scalar version of equation \eqref{IdnlsCMS} is a generalization and integrable perturbation of both the linear Schr\"odinger equation and the defocusing$\slash$focusing Calogero--Moser--Sutherland dNLS equation, introduced in Badreddine \cite{Badreddine2023} and G\'erard--Lenzmann \cite{Gerard-Lenzmann2022}.

    \item The spin Benjamin--Ono (sBO) equation on $\mathbb{T}$, introduced in Berntson--Langmann--Lenells \cite{sBOBLL2022},
\begin{equation}\label{SBO1}
\partial_t U =   \partial_x \left(|\mathrm{D}| U -  U^2\right) - i [U, |\mathrm{D}| U], \quad U=U(t,x) \in \mathbb{C}^{N \times N}, \quad (t,x) \in \mathbb{R}\times \mathbb{T}, \quad \mathrm{D}=-i \partial_x,
\end{equation}enjoys the following Lax pair structure,  
\begin{equation*}
\mathbf{L}_U^{\mathrm{sBO}}  := \mathrm{D}-\mathbf{T}_U^{(\mathbf{r})}, \quad \mathbf{B}_U^{\mathrm{sBO}}  :=  i \mathbf{T}_{|\mathrm{D}|U}^{(\mathbf{r})} - i\left( \mathbf{T}_{U}^{(\mathbf{r})} \right)^2 , \quad \forall U \in C^{\infty}(\mathbb{T}; \mathbb{C}^{N \times N}).
\end{equation*}thanks to P. G\'erard's work \cite{sBOLaxP2022}.
    \item The matrix KdV equation on $\mathbb{T}$
\begin{equation}\label{NKdV}
\partial_t U = 3 \partial_x (U^2) - \partial_x^3 U, \quad U=U(t,x) \in \mathbb{C}^{N \times N}, \quad (t,x) \in \mathbb{R} \times \mathbb{T}, 
\end{equation}enjoys the following Lax pair structure, thanks to the work Lax \cite{LaxPairCPAMKdV},
\begin{equation*}
\mathbf{L}_U^{\mathrm{KdV}}= U - \partial_x^2, \quad \mathbf{B}_U^{\mathrm{KdV}}=  - 4\partial_x^3 + 6U \partial_x + 3 (\partial_x U), \quad \forall U \in C^{\infty}(\mathbb{T}; \mathbb{C}^{N \times N}).
\end{equation*}
 
    \item The matrix cubic Schr\"odinger system on $\mathbb{T}$
\begin{equation}\label{MNLS}
\begin{cases}
i \partial_t U = -\partial_x^2 U + 2 UV^*U,\\
i \partial_t V =  -\partial_x^2 V + 2 VU^*V,\\
\end{cases}\quad U=U(t,x), \; V=V(t,x)\in \mathbb{C}^{M\times N},  \quad (t,x) \in \mathbb{R} \times \mathbb{T}, 
\end{equation}enjoys the following Lax pair structure, thanks to the work Zakharov--Shabat \cite{ZS1972}, 
\begin{equation*}
\mathbf{L}_{\left(  U, V\right)}^{\mathrm{NLS}}= \begin{pmatrix}
i \partial_x & U\\
V^* & -i \partial_x 
\end{pmatrix} , \; \mathbf{B}_{\left(  U, V\right)}^{\mathrm{NLS}}=\begin{pmatrix}
2 i \partial_x^2 -iUV^* & \partial_x U + 2U\partial_x\\
\partial_x V^* + 2V^*\partial_x & -2i \partial_x^2 +iV^*U  
\end{pmatrix}, \; \forall U, V \in C^{\infty}(\mathbb{T}; \mathbb{C}^{M \times N}).
\end{equation*}
\end{enumerate}
In previous examples, if the scalar multiplication is replaced by the right multiplication of matrices, then the Lax pair of the original scalar equation becomes the Lax pair of the corresponding matrix equation. It leads automatically to the following question.

\begin{ques}\label{ScamultoMaMul}
If we substitute the right multiplication of matrices for the scalar multiplication in the Lax pair when doing the matrix generalization for an \textbf{arbitrary} integrable equation, will this operation always give a 'Lax pair' for the corresponding matrix equation?
\end{ques}

\noindent The answer is \textbf{False} due to the matrix generalization of the cubic Szeg\H{o} equation from \eqref{sSzego} to \eqref{MSzego}. Recall that the scalar Hankel operator of symbol $u \in H^{\frac{1}{2}}_+(\mathbb{T}; \mathbb{C})$ is given by
\begin{equation}\label{Hankelsca}
H_u : f \in L^2_+(\mathbb{T}; \mathbb{C}) \mapsto H_u(f):=\Pi_{\geq 0}(u \bar{f}) \in L^2_+(\mathbb{T}; \mathbb{C})
\end{equation}The Hankel operator $H_u$ is a $\mathbb{C}$-antilinear Hilbert-Schmidt operator on $L^2_+(\mathbb{T}; \mathbb{C})$. The scalar Toeplitz operator of symbol $b \in L^{\infty} (\mathbb{T}; \mathbb{C})$, which is given by 
\begin{equation}\label{Toeplitzsca}
T_b : f \in L^2_+(\mathbb{T}; \mathbb{C}) \mapsto T_b(f) =\Pi_{\geq 0}(bf) \in L^2_+(\mathbb{T}; \mathbb{C}),
\end{equation}is a bounded $\mathbb{C}$-linear operator on $L^2_+(\mathbb{T}; \mathbb{C})$. If $u \in H^s_+(\mathbb{T}; \mathbb{C})$ for some $s> \frac{1}{2}$, set $B_u = \tfrac{i}{2}H_u^2 - i T_{|u|^2}$. According to G\'erard--Grellier \cite{GGANNENS}, $(H_u, B_u)$ is a Lax pair of \eqref{sSzego}, i.e. the function $u \in C^{\infty}(\mathbb{R}; H^s_+(\mathbb{T}; \mathbb{C}))$ solves \eqref{sSzego} if and only if 
\begin{equation}\label{LaxPScaSzego}
\tfrac{\mathrm{d}}{\mathrm{d}t} H_{u(t)} = [B_{u(t)}, H_{u(t)}].
\end{equation}When generalizing to the matrix Szeg\H{o} equation, the Hankel operator has two  different matrix generalizations. Given $d, M,N\in\mathbb{N}_+$, the left and right Hankel operators of symbol $U \in H^{\frac{1}{2}}_+(\mathbb{T}; \mathbb{C}^{M\times N})$ are defined by
\begin{equation}\label{rlHankelIntro}
\begin{split}
& \mathbf{H}^{(\mathbf{r})}_{U} : F \in L^2_+(\mathbb{T}; \mathbb{C}^{d \times N}) \mapsto \mathbf{H}^{(\mathbf{r})}_{U}(F) = \Pi_{\geq 0}(U F^*) \in L^2_+(\mathbb{T}; \mathbb{C}^{M \times d});\\
 & \mathbf{H}^{(\mathbf{l})}_{U} : G \in L^2_+(\mathbb{T}; \mathbb{C}^{M \times d}) \mapsto \mathbf{H}^{(\mathbf{l})}_{U}(G) = \Pi_{\geq 0}(G^* U) \in L^2_+(\mathbb{T}; \mathbb{C}^{d \times N}).\\
\end {split}
\end{equation}Assume that $M \ne N$, then $L^2_+(\mathbb{T}; \mathbb{C}^{d \times N}) \bigcap L^2_+(\mathbb{T}; \mathbb{C}^{M \times d}) = \emptyset$. So it is impossible to find $d \in \mathbb{N}_+$ and an operator $\mathbf{B}_U$ such that the Lie bracket $[\mathbf{B}_U, \mathbf{H}^{(\mathbf{r})}_{U} ] =\mathbf{B}_U \mathbf{H}^{(\mathbf{r})}_{U} -  \mathbf{H}^{(\mathbf{r})}_{U}\mathbf{B}_U$ is a  well defined operator from $L^2_+(\mathbb{T}; \mathbb{C}^{d \times N})$ to $L^2_+(\mathbb{T}; \mathbb{C}^{M \times d})$, according to the rules of matrix addition and multiplication. The similar result holds for $\mathbf{H}^{(\mathbf{l})}_{U}$. Consequently, neither $\mathbf{H}^{(\mathbf{r})}_{U}$ nor $\mathbf{H}^{(\mathbf{l})}_{U}$ can be a candidate for the Lax operator of the matrix Szeg\H{o} equation \eqref{MSzego}, while the single scalar Hankel operator $H_u$ is a Lax operator of the scalar Szeg\H{o} equation \eqref{sSzego}.  Unlike the previous integrable systems  \eqref{IdnlsCMS}, \eqref{SBO1},  \eqref{NKdV},  \eqref{MNLS}, the matrix Szeg\H{o} equation \eqref{MSzego} provides one counter-example of conjecture $\ref{ScamultoMaMul}$.\\

\noindent When generalizing a scalar equation to its matrix equation, the transpose transform 
\begin{equation}\label{Transpose}
\mathfrak{T}=\mathfrak{T}^{-1}: A \in \mathbb{C}^{M \times N} \mapsto \mathfrak{T}(A) = A^{\mathrm{T}} \in \mathbb{C}^{N \times M} 
\end{equation}becomes nontrivial if $(M,N) \ne (1,1)$. The matrix Szeg\H{o} equation \eqref{MSzego} is invariant under transposing, i.e. if  $U \in C^{\infty}(\mathbb{R}; H^s_+(\mathbb{T};  \mathbb{C}^{M\times N}))$ solves \eqref{MSzego}, so does $U^{\mathrm{T}} \in C^{\infty}(\mathbb{R}; H^s_+(\mathbb{T};  \mathbb{C}^{N\times M}))$, $\forall s >\frac{1}{2}$. In addition, the left Hankel operator  $\mathbf{H}^{(\mathbf{l})}_{U}$ is conjugate to the right Hankel operator  $\mathbf{H}^{(\mathbf{r})}_{U^{\mathrm{T}}}$ via the transpose transform, i.e.
\begin{equation}\label{TconjHrHl}
\mathbf{H}^{(\mathbf{l})}_{U} = \mathfrak{T}\circ \mathbf{H}^{(\mathbf{r})}_{U^{\mathrm{T}}} \circ\mathfrak{T}, \quad \mathbf{H}^{(\mathbf{r})}_{U} = \mathfrak{T}\circ \mathbf{H}^{(\mathbf{l})}_{U^{\mathrm{T}}} \circ\mathfrak{T}.
\end{equation}Even though one single matrix Hankel operator fails to be a Lax operator of \eqref{MSzego}, the matrix Szeg\H{o} equation \eqref{MSzego} still enjoys a Lax pair structure, which is provided by the double matrix Hankel operators $\mathbf{H}^{(\mathbf{r})}_{U} \mathbf{H}^{(\mathbf{l})}_{U}$ and  $\mathbf{H}^{(\mathbf{l})}_{U} \mathbf{H}^{(\mathbf{r})}_{U}$. Before stating the main results, we give the high regularity global wellposedness result of \eqref{MSzego}.

\begin{prop}\label{GWPH0.5}
Given $U_0 \in H^{\frac{1}{2}}_+(\mathbb{T};  \mathbb{C}^{M\times N})$, there exists a unique function $U \in C(\mathbb{R}; H^{\frac{1}{2}}_+(\mathbb{T};  \mathbb{C}^{M\times N}))$ solving the cubic matrix Szeg\H{o} equation \eqref{MSzego} such that $U(0)=U_0$. For each $T>0$, the flow map $\Phi : U_0 \in H^{\frac{1}{2}}_+(\mathbb{T};  \mathbb{C}^{M\times N}) \mapsto U\in C([-T, T]; H^{\frac{1}{2}}_+(\mathbb{T};  \mathbb{C}^{M\times N}))$ is continuous. Moreover, if $U_0 \in H^s_+(\mathbb{T};  \mathbb{C}^{M\times N})$ for some $s > \frac{1}{2}$, then $U \in C^{\infty}(\mathbb{R}; H^s_+(\mathbb{T};  \mathbb{C}^{M\times N}))$.
\end{prop}
 
\subsection{Main results}
\noindent Given $n \in \mathbb{Z}$, we set $\mathbf{e}_n : x \in \mathbb{T} \mapsto e^{inx} \in \mathbb{C}$. Given any positive integers $M,N\in\mathbb{N}_+$,  the shift operator $\mathbf{S}$ is defined by
\begin{equation}\label{Sdef}
\mathbf{S}: F \in L^2_+(\mathbb{T}; \mathbb{C}^{M\times N}) \mapsto \mathbf{e}_1 F \in L^2_+(\mathbb{T}; \mathbb{C}^{M\times N}).
\end{equation}Its $L^2_+$-adjoint is denoted by $\mathbf{S}^* \in \mathcal{B}\left(L^2_+(\mathbb{T}; \mathbb{C}^{M\times N})\right)$. If $F= \sum_{n \geq 0}\hat{F}(n)\mathbf{e}_n \in L^2_+(\mathbb{T}; \mathbb{C}^{M\times N})$, 
\begin{equation}\label{S*defIntro}
\mathbf{S}^* (F)= \Pi_{\geq 0}\left( \mathbf{e}_{-1} F\right) = \sum_{n \geq 0}\hat{F}(n+1)\mathbf{e}_n.
\end{equation}The left and right shifted Hankel operators of the symbol $U \in H^{\frac{1}{2}}_+(\mathbb{T}; \mathbb{C}^{M\times N})$ are defined by
\begin{equation}\label{rlShHanIntro}
\begin{split}
& \mathbf{K}^{(\mathbf{r})}_{U} = \mathbf{H}^{(\mathbf{r})}_{U}\mathbf{S} =   \mathbf{S}^* \mathbf{H}^{(\mathbf{r})}_{U} =\mathbf{H}^{(\mathbf{r})}_{\mathbf{S}^*  U}, \quad  \mathbf{K}^{(\mathbf{l})}_{U} = \mathbf{H}^{(\mathbf{l})}_{U}\mathbf{S} =   \mathbf{S}^* \mathbf{H}^{(\mathbf{l})}_{U} =\mathbf{H}^{(\mathbf{l})}_{\mathbf{S}^*  U}. 
\end{split}
\end{equation}The left Toeplitz operator of symbol $V \in L^2(\mathbb{T}; \mathbb{C}^{M \times N})$ is defined by 
\begin{equation}\label{Ltoep}
\begin{split}
\mathbf{T}^{(\mathbf{l})}_{V} : F \in H^1_+(\mathbb{T}; \mathbb{C}^{d \times M}) \mapsto \mathbf{T}^{(\mathbf{l})}_{V}(F) = \Pi_{\geq 0}(FV) \in L^2_+(\mathbb{T}; \mathbb{C}^{d \times N}), \quad \forall d \in \mathbb{N}_+. 
\end {split}
\end{equation}The first result of this paper gives the two-Lax-pair structure of the matrix Szeg\H{o} equation \eqref{MSzego}.
\begin{thm}\label{LaxPairThm}
Given $M,N,d \in \mathbb{N}_+$ and $s>\tfrac{1}{2}$, if $U \in C^{\infty}(\mathbb{R}; H^s_+(\mathbb{T};  \mathbb{C}^{M\times N}))$ solves the matrix Szeg\H{o} equation \eqref{MSzego}, then the time-dependent operators  $\mathbf{H}^{(\mathbf{r})}_{U}\mathbf{H}^{(\mathbf{l})}_{U}, \mathbf{K}^{(\mathbf{r})}_{U}\mathbf{K}^{(\mathbf{l})}_{U} \in C^{\infty}(\mathbb{R}; \mathcal{B}(L^2_+(\mathbb{T};  \mathbb{C}^{M\times d})))$ and  $\mathbf{H}^{(\mathbf{l})}_{U}\mathbf{H}^{(\mathbf{r})}_{U}, \mathbf{K}^{(\mathbf{l})}_{U}\mathbf{K}^{(\mathbf{r})}_{U} \in C^{\infty}(\mathbb{R}; \mathcal{B}(L^2_+(\mathbb{T};  \mathbb{C}^{d\times N})))$ satisfy the following identities:
\begin{equation}\label{4HeiLaxMSzego}
\begin{split}
& \tfrac{\mathrm{d}}{\mathrm{d}t}(\mathbf{H}^{(\mathbf{r})}_{U(t)}\mathbf{H}^{(\mathbf{l})}_{U(t)}) = i[\mathbf{H}^{(\mathbf{r})}_{U(t)}\mathbf{H}^{(\mathbf{l})}_{U(t)}, \; \mathbf{T}^{(\mathbf{r})}_{U(t) U(t)^*}]; \quad \tfrac{\mathrm{d}}{\mathrm{d}t}(\mathbf{H}^{(\mathbf{l})}_{U(t)}\mathbf{H}^{(\mathbf{r})}_{U(t)}) = i[\mathbf{H}^{(\mathbf{l})}_{U(t)}\mathbf{H}^{(\mathbf{r})}_{U(t)},  \; \mathbf{T}^{(\mathbf{l})}_{U(t)^* U(t)}]; \\
& \tfrac{\mathrm{d}}{\mathrm{d}t}(\mathbf{K}^{(\mathbf{r})}_{U(t)}\mathbf{K}^{(\mathbf{l})}_{U(t)}) = i[\mathbf{K}^{(\mathbf{r})}_{U(t)}\mathbf{K}^{(\mathbf{l})}_{U(t)}, \; \mathbf{T}^{(\mathbf{r})}_{U(t) U(t)^*}]; \quad \tfrac{\mathrm{d}}{\mathrm{d}t}(\mathbf{K}^{(\mathbf{l})}_{U(t)}\mathbf{K}^{(\mathbf{r})}_{U(t)}) = i[\mathbf{K}^{(\mathbf{l})}_{U(t)}\mathbf{K}^{(\mathbf{r})}_{U(t)},  \; \mathbf{T}^{(\mathbf{l})}_{U(t)^* U(t)}]. 
\end{split}
\end{equation}
\end{thm}
\begin{rem}
In the scalar case, i.e. $M=N=1$, the transpose transform $\mathfrak{T}$ becomes trivial and the right Hankel operator $\mathbf{H}^{(\mathbf{r})}_{U}$ coincides with the left Hankel operator $\mathbf{H}^{(\mathbf{l})}_{U}$. In that case, the single Hankel operator $H_u$ becomes a Lax operator of the cubic scalar Szeg\H{o} equation \eqref{sSzego}.
\end{rem}

\noindent Thanks to the unitary equivalence between $\mathbf{H}^{(\mathbf{r})}_{U(t)}\mathbf{H}^{(\mathbf{l})}_{U(t)}$ and $\mathbf{H}^{(\mathbf{r})}_{U(0)}\mathbf{H}^{(\mathbf{l})}_{U(0)}$, we have the following large time estimate for the high regularity Sobolev norms of the solution to \eqref{MSzego}.
\begin{cor}
 There exists a constant $\mathcal{C}_s >0$ such that if $U \in C^{\infty}(\mathbb{R}; H^s_+(\mathbb{T};  \mathbb{C}^{M\times N}))$ solves  \eqref{MSzego} for some $s>1$, then $ \sup_{t \in \mathbb{R}}\|U(t)\|_{L^{\infty}(\mathbb{T}; \mathbb{C}^{M \times N})} \leq 2 \mathrm{Tr}\sqrt{\mathbf{H}^{(\mathbf{r})}_{U(0)}\mathbf{H}^{(\mathbf{l})}_{U(0)}} \leq \mathcal{C}_s \|U(t)\|_{H^{s}(\mathbb{T}; \mathbb{C}^{M \times N})}$ and
\begin{equation}\label{AtmostExpGrowth}
 \sup_{t \in \mathbb{R}}\|U(t)\|_{H^{s}(\mathbb{T}; \mathbb{C}^{M \times N})} \leq \mathcal{C}_s   e^{\mathcal{C}_s  |t| \|U(0)\|_{H^{s}(\mathbb{T}; \mathbb{C}^{M \times N})} } \|U(0)\|_{H^{s}(\mathbb{T}; \mathbb{C}^{M \times N})}.
\end{equation}
\end{cor}

\noindent Given any positive integers $M,N\in\mathbb{N}_+$, the integral operator is defined by
\begin{equation}\label{IntdefIntro}
\mathbf{I} : F \in  L^1 (\mathbb{T}; \mathbb{C}^{M\times N}) \mapsto \mathbf{I}(F)= \hat{F}(0) =\frac{1}{2\pi} \int_{0}^{2 \pi} F(x) \mathrm{d}x  \in \mathbb{C}^{M \times N}.
\end{equation}Inspired from the works G\'erard--Grellier \cite{GGTAMS}, G\'erard \cite{BOExp2022} and  Badreddine \cite{Badreddine2023}, we compare three families of unitary operators acting on the shift operator $\mathbf{S}^* \in \mathcal{B}\left(L^2_+(\mathbb{T}; \mathbb{C}^{M\times N})\right)$ by conjugation in order to linearize the Hamiltonian flow of \eqref{MSzego} and obtain an explicit expression of the Poisson integral of every $H^{\frac{1}{2}}_+$-solution. This explicit formula is given by the following theorem.

\begin{thm}\label{ExpForThm}
Given $M,N \in \mathbb{N}_+$, if $U : t \in \mathbb{R} \mapsto U(t)=\sum_{n \geq 0}\hat{U}(t,n)\mathbf{e}_n\in   H^{\frac{1}{2}}_+(\mathbb{T};  \mathbb{C}^{M\times N})$ solves the matrix Szeg\H{o} equation \eqref{MSzego} with initial data $U(0)=U_0 \in H^{\frac{1}{2}}_+(\mathbb{T};  \mathbb{C}^{M\times N})$, then
\begin{equation}\label{mszeExpFor}
\begin{split}
\hat{U}(t,n) = & \mathbf{I}\left( (e^{-it  \mathbf{H}^{(\mathbf{r})}_{U_0}\mathbf{H}^{(\mathbf{l})}_{U_0}} e^{it  \mathbf{K}^{(\mathbf{r})}_{U_0}\mathbf{K}^{(\mathbf{l})}_{U_0}} \mathbf{S}^*)^n e^{-it  \mathbf{H}^{(\mathbf{r})}_{U_0}\mathbf{H}^{(\mathbf{l})}_{U_0}}  (U_0) \right) \\
= & \mathbf{I}\left( (e^{-it  \mathbf{H}^{(\mathbf{l})}_{U_0}\mathbf{H}^{(\mathbf{r})}_{U_0}} e^{it  \mathbf{K}^{(\mathbf{l})}_{U_0}\mathbf{K}^{(\mathbf{r})}_{U_0}} \mathbf{S}^*)^n e^{-it  \mathbf{H}^{(\mathbf{l})}_{U_0}\mathbf{H}^{(\mathbf{r})}_{U_0}}  (U_0) \right) \in \mathbb{C}^{M \times N}.
\end{split}
\end{equation}
\end{thm} 
\noindent Since the single Hankel operator is no longer a Lax operator, some steps in the proof of Theorem 1 of G\'erard--Grellier  \cite{GGTAMS} needs to be modified in order to prove theorem $\ref{ExpForThm}$. Thanks to the Kronecker theorem, the right Hankel operator $\mathbf{H}^{(\mathbf{r})}_{U}$ and the double Hankel operator $\mathbf{H}^{(\mathbf{r})}_{U} \mathbf{H}^{(\mathbf{l})}_{U}$ have the same image, when the symbol $U$ is rational. Then we start to prove \eqref{mszeExpFor} for rational solutions and complete the proof by density argument. \\ 

\noindent This paper is organized as follows. We recall matrix-valued functional spaces and inequalities in section $\ref{SecPre}$. Section $\ref{SecLaxP}$ is dedicated to establishing the Lax pair structure of \eqref{MSzego} and proving theorem $\ref{LaxPairThm}$. The explicit formula is obtained in section $\ref{SecExplicit}$.

\section{Preliminaries}\label{SecPre}
\noindent We give some preliminaries of the matrix valued functional spaces and prove proposition $\ref{GWPH0.5}$. Given $p \in [1, +\infty]$,  $s \geq 0$ and $M,N \in \mathbb{N}_+$, a matrix function $U=\left(U_{kj}\right)_{1 \leq k \leq M, 1 \leq j \leq N}$ belongs to $L^p(\mathbb{T}; \mathbb{C}^{M \times N})$ if and only if its $kj$-entry $U_{kj}$ belongs to $L^p(\mathbb{T}; \mathbb{C})$. We set
\begin{equation}\label{LpMat}
\begin{split}
& \|U\|_{L^p(\mathbb{T}; \mathbb{C}^{M \times N})}^2 := \sum_{k=1}^M \sum_{j=1}^N \|U_{kj}\|_{L^p(\mathbb{T}; \mathbb{C})}^2 = \|U^*\|_{L^p(\mathbb{T}; \mathbb{C}^{N \times M})}^2 ; \\
 & \|U\|_{H^s(\mathbb{T}; \mathbb{C}^{M \times N})}^2 := \sum_{k=1}^M \sum_{j=1}^N \|U_{kj}\|_{H^s(\mathbb{T}; \mathbb{C})}^2 =\|U^*\|_{H^s(\mathbb{T}; \mathbb{C}^{N \times M})}^2 .
\end{split}
\end{equation}Let $H^s(\mathbb{T}; \mathbb{C}^{M\times N})$ denote the matrix-valued Sobolev space, i.e. 
\begin{equation}\label{M-Hs}
H^s(\mathbb{T}; \mathbb{C}^{M\times N}) ):= \{U = \sum_{k=1}^M\sum_{j=1}^N U_{kj} \mathbb{E}_{{kj}}^{(MN)}  : U_{kj} \in H^s(\mathbb{T};  \mathbb{C}), \; \forall 1\leq k \leq M,  \;
1 \leq j \leq N \}.
\end{equation}Then $L^2(\mathbb{T}; \mathbb{C}^{M\times N}) = H^0(\mathbb{T}; \mathbb{C}^{M\times N})$. Equipped with the following inner product
\begin{equation}\label{InnerPro}
(U, V) \in L^2(\mathbb{T};  \mathbb{C}^{M\times N})^2 \mapsto \langle U, V \rangle_{L^2(\mathbb{T};  \mathbb{C}^{M\times N}) } := \frac{1}{2\pi}\int_0^{2\pi}\mathrm{tr} \left( U(x) V(x)^{*}\right)\mathrm{d}x \in \mathbb{C},
\end{equation}$L^2(\mathbb{T};  \mathbb{C}^{M\times N})$ is a $\mathbb{C}$-Hilbert space. Given a function  $U = \left(U_{kj}\right)_{1 \leq k, j \leq d} \in L^2(\mathbb{T}; \mathbb{C}^{M \times N})$, its Fourier expansion is given as follows,
\begin{equation}
U(x)= \sum_{n \in \mathbb{Z}}\hat{U}(n)e^{inx}, \quad \hat{U}(n)=\frac{1}{2\pi}\int_{0}^{2\pi}U(x)e^{-inx} \mathrm{d}x= \sum_{k=1}^M\sum_{j=1}^N \hat{U}_{kj}(n) \mathbb{E}_{{kj}}^{(MN)}  \in \mathbb{C}^{M\times N}.
\end{equation}The Parseval's formula reads as
\begin{equation}\label{Parseval}
\langle U, V\rangle_{L^2} = \sum_{k=1}^M\sum_{j=1}^N \int_{\mathbb{R}}U_{kj}(x) \overline{V_{kj}(x)} \mathrm{d}x =  \sum_{n \in \mathbb{Z}} \sum_{k=1}^M\sum_{j=1}^N\hat{U}_{kj}(n) \overline{\hat{V}_{kj}(n)} 
= \sum_{n \in \mathbb{Z}} \mathrm{tr} \left(\hat{U}(n)\left(\hat{V}(n)\right)^* \right).
\end{equation}The negative Szeg\H{o} projector  $\Pi_{< 0} = \mathrm{id}_{L^2(\mathbb{T};  \mathbb{C}^{M\times N})} - \Pi_{\geq 0}$ on $L^2(\mathbb{T};  \mathbb{C}^{M\times N})$ is given by
\begin{equation}\label{Szego-}
 \Pi_{<0}(U)  := \sum_{n < 0}\hat{U}(n) \mathbf{e}_n \in  L^2(\mathbb{T};  \mathbb{C}^{M\times N}), \quad \forall U= \sum_{n \in \mathbb{Z}}\hat{U}(n)\mathbf{e}_n \in  L^2(\mathbb{T};  \mathbb{C}^{M\times N}).
 \end{equation}The filtered $H^s$-spaces are given by 
 \begin{equation}
 H^s_+(\mathbb{T};  \mathbb{C}^{M\times N}):=  \Pi_{\geq 0}\left(H^s(\mathbb{T};  \mathbb{C}^{M\times N}) \right), \quad  H^s_-(\mathbb{T};  \mathbb{C}^{M\times N}):=  \Pi_{<0}\left(H^s(\mathbb{T};  \mathbb{C}^{M\times N}) \right)
 \end{equation}Then the $\mathbb{C}$-Hilbert space $ L^2(\mathbb{T};  \mathbb{C}^{M\times N})$ has the following orthogonal decomposition
 \begin{equation}
 L^2(\mathbb{T};  \mathbb{C}^{M\times N}) = L^2_+(\mathbb{T};  \mathbb{C}^{M\times N}) \bigoplus L^2_-(\mathbb{T};  \mathbb{C}^{M\times N}), \quad L^2_+(\mathbb{T};  \mathbb{C}^{M\times N}) \perp L^2_-(\mathbb{T};  \mathbb{C}^{M\times N}).
 \end{equation}For any $U=\sum_{n\in \mathbb{Z}}\hat{U}(n)\mathbf{e}_n\in L^2(\mathbb{T}; \mathbb{C}^{M\times N})$, where $\mathbf{e}_n: x \in \mathbb{T}\mapsto e^{inx}\in \mathbb{C}$, then  
\begin{equation}\label{Pi<0adjTorus}
\Pi_{<0} U = \left(\Pi_{\geq 0}(U^*) \right)^* - \hat{U}(0) \in L^2_-(\mathbb{T}; \mathbb{C}^{M\times N}), \quad \left(\Pi_{\geq 0} U \right)^* = \Pi_{<0}(U^*) + ( \hat{U}(0) )^* \in L^2(\mathbb{T}; \mathbb{C}^{N\times M}).
\end{equation}

\begin{lem}
If $U \in L^2(\mathbb{T}; \mathbb{C}^{M\times N})$, $A\in \mathbb{C}^{N \times P}$, $B\in \mathbb{C}^{Q\times M}$ for some $M,N,P,Q \in \mathbb{N}_+$, then
\begin{equation}\label{matrixXL^2}
\begin{split}
& \Pi_{\geq 0} (UA) = \left(\Pi_{\geq 0}U \right)A \in L^2_+ (\mathbb{T}; \mathbb{C}^{M \times P}); \quad \Pi_{< 0} (UA) = \left(\Pi_{< 0}U\right)A \in L^2_- (\mathbb{T}; \mathbb{C}^{M \times P}); \\
& \Pi_{\geq 0} (BU) = B \left(\Pi_{\geq 0}U \right) \in  L^2_+ (\mathbb{T}; \mathbb{C}^{Q \times N}); \quad  \Pi_{< 0} (BU) = B \left(\Pi_{< 0}U \right) \in  L^2_- (\mathbb{T}; \mathbb{C}^{Q \times N}).
\end{split}
\end{equation}
\end{lem}

\begin{proof}
Since $A_{nj}\in \mathbb{C}$, then $\left(\Pi_{\geq 0} (UA)\right)_{kj}=\sum_{n=1}^N(\Pi_{\geq 0}U_{kn})A_{nj} = \left((\Pi_{\geq 0}U)A \right)_{kj}$, if $1 \leq k \leq M$, $1 \leq j \leq P$. Since $B_{sm} \in \mathbb{C}$, then $\left(\Pi_{\geq 0} (BU)\right)_{st}=\sum_{m=1}^N B_{sm} \Pi_{\geq 0}U_{mt} = \left(B(\Pi_{\geq 0}U) \right)_{st}$, if $1 \leq s \leq Q$, $1 \leq t \leq N$.
\end{proof}

\begin{lem}\label{AB*L2-}
Given $A \in L^2_-(\mathbb{T}; \mathbb{C}^{M\times N})$ and $B \in L^2_+(\mathbb{T}; \mathbb{C}^{d \times N})$ for some $M,N, d \in \mathbb{N}_+$, if one of $A, B$ is essentially bounded , then $AB^* \in L^2_-(\mathbb{T}; \mathbb{C}^{M \times d})$.
\end{lem}
\begin{proof}
If $A = \sum_{n \geq 0}\hat{A}(n)\mathbf{e}_n$, $B = \sum_{n \geq 0}\hat{B}(n)\mathbf{e}_n$,  $AB^*= \sum_{l \leq -1} \left(\sum_{n=l}^{-1}\hat{A}(n)(\hat{B}(n-l))^*\right)\mathbf{e}_l\in L^2_-.$
\end{proof}

\begin{lem}\label{A*BL2-}
Given $A \in L^2_+(\mathbb{T}; \mathbb{C}^{M\times N})$ and $B \in L^2_-(\mathbb{T}; \mathbb{C}^{M \times d})$ for some $M,N, d \in \mathbb{N}_+$, if one of $A, B$ is essentially bounded , then $A^* B \in L^2_-(\mathbb{T}; \mathbb{C}^{N \times d})$.
\end{lem}
\begin{proof}
If $A = \sum_{n \geq 0}\hat{A}(n)\mathbf{e}_n$, $B = \sum_{n \geq 0}\hat{B}(n)\mathbf{e}_n$,  $A^* B= \sum_{l\leq -1} \left(\sum_{n=l}^{-1}(\hat{A}(n-l))^*\hat{B}(n) \right)\mathbf{e}_l\in L^2_-.$
\end{proof}

\begin{lem}\label{ABL2+}
Given $A \in L^2_+(\mathbb{T}; \mathbb{C}^{M\times N})$ and $B \in L^2_+(\mathbb{T}; \mathbb{C}^{N \times d})$ for some $M,N, d \in \mathbb{N}_+$, if one of $A, B$ is essentially bounded , then $A  B \in L^2_+(\mathbb{T}; \mathbb{C}^{M \times d})$.
\end{lem}
\begin{proof}
If $A = \sum_{n \geq 0}\hat{A}(n)\mathbf{e}_n$, $B = \sum_{n \geq 0}\hat{B}(n)\mathbf{e}_n$, then $AB= \sum_{l \geq 0} \left(\sum_{n=0}^l \hat{A}(n)\hat{B}(l-n)\right)\mathbf{e}_l\in L^2_+.$
\end{proof}

\begin{prop}[Cauchy]\label{CauchyThmODE}
Let $\mathcal{E}$ be a Banach space, $\mathcal{I}$ is an open interval of $\mathbb{R}$ and $A \in C^0 (\mathcal{I}; \mathcal{B}(\mathcal{E}))$, if $(t_0, x_0) \in \mathcal{I} \times \mathcal{E}$, there exists a unique function $x \in C^1(\mathcal{I}; \mathcal{B}(\mathcal{E}))$ such that $x(t_0)=x_0$ and
\begin{equation}
\tfrac{\mathrm{d}}{\mathrm{d}t} x(t) = A(t)\left( x(t)\right), \quad \forall t \in \mathcal{I}.
\end{equation} 
\end{prop}
\begin{proof}
See Theorem 1.1.1 of Chemin \cite{CheminNoteEDPM2-2016}.
\end{proof}

\subsection{High regularity wellposedness}
\noindent If $s\geq \frac{1}{2}$, the proof of the $H^s$-global wellposedness of the matrix Szeg\H{o} equation follows directly the steps of section 2 of G\'erard--Grellier \cite{GGANNENS} and the following matrix inequalities.

\begin{lem}[Brezis--Gallou\"et \cite{Brezis-Gallouet1980}]
 If $s > \frac{1}{2}$, there exists a constant $\mathcal{C}_s > 0$ such that $\forall U \in H^s(\mathbb{T}; \mathbb{C}^{M \times N})$ for some $M,N \in \mathbb{N}_+$, the following inequality holds,
\begin{equation}\label{BGMatrixIne}
\|U\|_{L^{\infty}(\mathbb{T}; \mathbb{C}^{M \times N})}^2 \leq \mathcal{C}_s^2 \|U\|_{H^{\frac{1}{2}}(\mathbb{T}; \mathbb{C}^{M \times N})}^2 \ln \left(2 + \tfrac{\|U\|_{H^s(\mathbb{T}; \mathbb{C}^{M \times N})}}{\|U\|_{H^{\frac{1}{2}}(\mathbb{T}; \mathbb{C}^{M \times N})}} \right).
\end{equation}
\end{lem}
\begin{proof}
If $U  \in H^s(\mathbb{T}; \mathbb{C}^{M \times N})\backslash \{0\}$ for some $s > \frac{1}{2}$, there exists $m\geq 1$ such that $m^{s - \frac{1}{2}} \sqrt{\ln(m+1)} \leq \tfrac{\|U\|_{H^s(\mathbb{T}; \mathbb{C}^{M \times N})}}{\|U\|_{H^{\frac{1}{2}}(\mathbb{T}; \mathbb{C}^{M \times N})}} \leq (m+1)^{s - \frac{1}{2}} \sqrt{\ln(m+2)}$. Set $\mathcal{A}_s := (s-\tfrac{1}{2})^{-1} +1$, so $(2+m^{s -\frac{1}{2}} \sqrt{\ln 2})^{2\mathcal{A}_s} \geq 2+m$. Appendix 2 in Page 805 of G\'erard--Grellier \cite{GGANNENS} yields that there exists $\mathcal{C}_s^{(1)}>0$ such that
\begin{equation*}
\begin{split}
& \|U\|_{L^{\infty}(\mathbb{T}; \mathbb{C}^{M \times N})}^2 \leq  \mathcal{C}_s^{(1)} \left(\|U\|_{H^{\frac{1}{2}}(\mathbb{T}; \mathbb{C}^{M \times N})}^2 \ln(2+m) + (m+1)^{1 -2 s } \|U\|_{H^{s}(\mathbb{T}; \mathbb{C}^{M \times N})}^2\right)\\
\leq & 4 \mathcal{A}_s \mathcal{C}_s^{(1)} \|U\|_{H^{\frac{1}{2}}(\mathbb{T})} \ln(2+m^{s -\frac{1}{2}} \sqrt{\ln 2}) \lesssim_s \|U\|_{H^{\frac{1}{2}}(\mathbb{T}; \mathbb{C}^{M \times N})} \ln \left(2+\tfrac{\|U\|_{H^s(\mathbb{T}; \mathbb{C}^{M \times N})}}{\|U\|_{H^{\frac{1}{2}}(\mathbb{T}; \mathbb{C}^{M \times N})}} \right).
\end{split}
\end{equation*}
\end{proof}
\begin{lem}[Trudinger \cite{Trudinger1967}]
There exists a universal constant $\mathcal{C} > 0$ such that $\forall U \in H^{\frac{1}{2}}(\mathbb{T}; \mathbb{C}^{M \times N})$ for some $M,N \in \mathbb{N}_+$, $\forall p \in [1, +\infty)$, the following inequality holds,
\begin{equation}\label{TruMatrixIne}
\|U\|_{L^p(\mathbb{T}; \mathbb{C}^{M \times N})}  \leq \mathcal{C} \sqrt{p} \|U\|_{H^{\frac{1}{2}}(\mathbb{T}; \mathbb{C}^{M \times N})}.
\end{equation}
\end{lem}
\begin{proof}
It is enough to plug Appendix 3 of G\'erard--Grellier \cite{GGANNENS}  into \eqref{LpMat}.
\end{proof}
 
\subsection{The Hamiltonian formalism}
\noindent The inner product of $\mathbb{C}$-Hilbert space $L^2_+(\mathbb{T}; \mathbb{C}^{M \times N})$ provides the canonical symplectic structure
\begin{equation}\label{SymplecticForm}
\omega(F, G):= \mathrm{Im}\langle F, G \rangle_{L^2(\mathbb{T}; \mathbb{C}^{M \times N})} = \mathrm{Im}\int_0^{2\pi} \mathrm{tr}(F(x) G(x)^*) \tfrac{\mathrm{d}x}{2\pi}, \quad \forall F, G \in L^2_+(\mathbb{T}; \mathbb{C}^{M \times N}),
\end{equation}because the mapping $\Upsilon^{\omega}: F \in L^2_+(\mathbb{T}; \mathbb{C}^{M \times N}) \mapsto F \lrcorner \omega \in \mathcal{B}(L^2_+(\mathbb{T}; \mathbb{C}^{M \times N}); \mathbb{R})$ is invertible, thanks to Riesz--Fr\'echet theorem. Given any smooth function $f: L^2_+(\mathbb{T}; \mathbb{C}^{M \times N}) \to \mathbb{R}$, its Fr\'echet derivative is denoted by $\nabla_U f$, its Hamiltonian vector field is given by $X_f:=-\left(\Upsilon^{\omega}\right)^{-1}(\mathrm{d}f)$, i.e.
\begin{equation}\label{HamVecFie}
\mathrm{d}f(U)(F)= \mathrm{Re}\langle F, \nabla_U f(U) \rangle_{L^2(\mathbb{T}; \mathbb{C}^{M \times N})} = \omega (F, X_f(U)), \quad \forall F, U\in L^2_+(\mathbb{T}; \mathbb{C}^{M \times N}).
\end{equation}Given another smooth function $g: L^2_+(\mathbb{T}; \mathbb{C}^{M \times N}) \to \mathbb{R}$, the Poisson bracket between $f$ and $g$ is given by 
\begin{equation}\label{PoisBrac}
\{f,g\}(U):= \omega(X_f(U), X_g(U))  = \mathrm{Im}\langle \nabla_U f(U), \nabla_U g(U) \rangle_{L^2(\mathbb{T}; \mathbb{C}^{M \times N})}, \quad \forall U\in L^2_+(\mathbb{T}; \mathbb{C}^{M \times N}).
\end{equation}Then  \eqref{MSzego} has the following Hamiltonian formalism,
\begin{equation}
\partial_t U(t)=-i \Pi_{\geq 0}(U(t) U(t)^* U(t))= X_{\mathbf{E}}(U(t)), \quad U(t) \in H^s_+(\mathbb{T}; \mathbb{C}^{M \times N}),  \quad s> \tfrac{1}{2},
\end{equation}where the energy functional $\mathbf{E}:   L^4_+(\mathbb{T}; \mathbb{C}^{M \times N}):= \Pi_{\geq 0}(L^4 (\mathbb{T}; \mathbb{C}^{M \times N})) \to \mathbb{R}$ is given by
\begin{equation}\label{EnergyFunc}
 \mathbf{E}(U):= \frac{1}{8\pi}\int_0^{2\pi} \mathrm{tr}\left( \left(U(x) U(x)^*\right)^2\right) \mathrm{d}x = \frac{1}{4}\|U U^*\|^2_{L^2(\mathbb{T}; \mathbb{C}^{M \times M})}, \quad \forall U \in L^4_+(\mathbb{T}; \mathbb{C}^{M \times N}).
\end{equation}The matrix Szeg\H{o} equation \eqref{MSzego} is invariant under both phase rotation and space translation, leading the following conservation laws by Noether's theorem,
\begin{equation}\label{ChargeMomentum}
\mathbf{q}(U)= \|U\|^2_{L^2_+(\mathbb{T}; \mathbb{C}^{M \times N})}, \quad  \mathfrak{j} (U)= \||\mathrm{D}|^{\frac{1}{2}} U\|^2_{L^2_+(\mathbb{T}; \mathbb{C}^{M \times N})} = -i\langle \partial_x U, U \rangle_{L^2_+(\mathbb{T}; \mathbb{C}^{M \times N})}.  
\end{equation}We have $\{\mathbf{q}, \mathbf{E}\} = \{\mathbf{q},  \mathfrak{j}\}= \{\mathbf{E},  \mathfrak{j}\}=0$ on $H^1_+(\mathbb{T}; \mathbb{C}^{M \times N})$.

\subsection{The Poisson integral and shift operators}
Given $n \in \mathbb{Z}$, we recall that
\begin{equation}\label{expfunTorus}
\mathbf{e}_n : x \in \mathbb{T} \mapsto e^{inx} \in \mathbb{C}.
\end{equation}Let $\mathbb{E}_{kj}^{(M N)} \in \mathbb{C}^{M \times N}$ denotes the $M\times N$ matrix whose $kj$-entry is $1$ and the other entries are all $0$. Given any $F=(F_{st})_{1\leq s \leq M, \; 1\leq t \leq N} \in L^1(\mathbb{T}; \mathbb{C}^{M\times N})$, we have $\langle F, \; \mathbb{E}_{kj}^{(M N)}\rangle_{L^2(\mathbb{T}; \mathbb{C}^{M\times N})} =\hat{F}_{kj}(0)$ and 
\begin{equation}\label{I<>formula}
\mathbf{I}(F) = \sum_{k=1}^M \sum_{j=1}^N \langle F, \; \mathbb{E}_{kj}^{(M N)}\rangle_{L^2(\mathbb{T}; \mathbb{C}^{M\times N})}\mathbb{E}_{kj}^{(M N)} \in \mathbb{C}^{M\times N}.
\end{equation}If $s \geq 0$, both the shift operator $\mathbf{S}$ and its adjoint $\mathbf{S}^*$ are bounded operators on $H^s_+(\mathbb{T}; \mathbb{C}^{M \times N})$. In fact, 
\begin{equation}
\|\mathbf{S}\|_{B\left(L^2_+(\mathbb{T}; \mathbb{C}^{M \times N})\right)} = 1, \qquad \|\mathbf{S}^*\|_{B\left( H^s_+(\mathbb{T}; \mathbb{C}^{M \times N})\right)} \leq 1,\quad \forall s \geq 0.
\end{equation}Then we have 
\begin{equation}\label{inverseSS*}
\mathbf{S}^* \mathbf{S} = \mathrm{id}_{L^2_+(\mathbb{T}; \mathbb{C}^{M \times N})}, \quad \mathbf{S} \mathbf{S}^*  = \mathrm{id}_{L^2_+(\mathbb{T}; \mathbb{C}^{M \times N})} - \mathbf{I}.
\end{equation}If $A \in  L^2 (\mathbb{T}; \mathbb{C}^{M\times N})$, then $\mathbf{e}_{-1} \Pi_{<0}A \in L^2_- (\mathbb{T}; \mathbb{C}^{M\times N})$. If $F= \sum_{n \geq 0}\hat{F}(n)\mathbf{e}_n \in L^2_+(\mathbb{T}; \mathbb{C}^{M\times N})$, then we use mathematical induction to deduce that 
\begin{equation}\label{S*mPower}
\left( \mathbf{S}^*\right)^m  (F)= \Pi_{\geq 0} \left( \mathbf{e}_{-m} F \right), \quad \hat{F}(m)=\mathbf{I}\left(\left( \mathbf{S}^*\right)^m  (F) \right),\quad \forall m \in \mathbb{N}. 
\end{equation}If $z=r e^{i\theta}$ for some $r \in [0, 1)$, $\theta \in \mathbb{T}$, the Poisson integral of $F= \sum_{n \geq 0}\hat{F}(n)\mathbf{e}_n \in L^2_+(\mathbb{T}; \mathbb{C}^{M\times N})$ is given by
\begin{equation}\label{PoiInt}
\underline{F}(z) = \mathscr{P}[F](r, \theta) : = \frac{1}{2\pi} \int_{-\pi}^{\pi}\mathbf{p}_r(\theta -x)F(x) \mathrm{d}x =\sum_{n\geq 0}z^n \hat{F}(n) \in \mathbb{C},
\end{equation}where $\mathbf{p}_r(x)=\frac{1-r^2}{1-2r \cos(x) +r^2}=\sum_{n \in \mathbb{Z}}r^{|n|}e^{i n x}$, $\forall  x \in \mathbb{T}$, denotes the Poisson kernel on the torus. Then Theorem 11.16 of Rudin  \cite{RudinRACA} yields that $\mathscr{P}[F](r) =  \sum_{n\geq 0}r^n \hat{F}(n) \mathbf{e}_n \in L^2_+(\mathbb{T}; \mathbb{C}^{M\times N})$ and
\begin{equation}\label{L2convPoiInt}
\sup_{0 \leq r <1}\|\mathscr{P}[F](r)\|_{L^2_+(\mathbb{T}; \mathbb{C}^{M\times N})} \leq \|F\|_{L^2_+}, \quad  \lim_{r \to 1^-}\|\mathscr{P}[F](r) -F\|_{L^2_+(\mathbb{T}; \mathbb{C}^{M\times N})}=0.
\end{equation}In addition, if $U= \sum_{n \geq 0}\hat{U}(n)\mathbf{e}_n \in C^0 \bigcap L^2_+(\mathbb{T}; \mathbb{C}^{M\times N})$, then we have  
\begin{equation}\label{LinftyconvPoiInt}
\sup_{0 \leq r <1}\|\mathscr{P}[F](r)\|_{L^{\infty} (\mathbb{T}; \mathbb{C}^{M\times N})} \leq \|F\|_{L^{\infty}}, \quad  \lim_{r \to 1^-}\|\mathscr{P}[F](r) -F\|_{L^{\infty}(\mathbb{T}; \mathbb{C}^{M\times N})}=0.
\end{equation}by Theorem 11.8 and 11.16 of Rudin  \cite{RudinRACA}. 

\begin{lem}
Given any $M,N\in\mathbb{N}_+$ and $z\in \mathbb{C}$ such that $|z| <1$, if $U \in L^2_+(\mathbb{T}; \mathbb{C}^{M\times N})$, then
\begin{equation}\label{InvForTorus}
\underline{U}(z) = \mathbf{I} \left( (\mathrm{id}_{L^2_+(\mathbb{T}; \mathbb{C}^{M \times N})}- z \mathbf{S}^* )^{-1} U \right).
\end{equation}
\end{lem}
\begin{proof}
If $U= \sum_{n \geq 0}\hat{U}(n)\mathbf{e}_n \in  L^2_+(\mathbb{T}; \mathbb{C}^{M\times N})$, formulas \eqref{PoiInt}, \eqref{S*mPower} and Theorem 18.3 of Rudin \cite{RudinRACA} yield that $\underline{U}(z) = \sum_{n\geq 0}  \mathbf{I}\left(\left( z\mathbf{S}^*\right)^m  (U) \right) =\mathbf{I} ( (\mathrm{id}_{L^2_+(\mathbb{T}; \mathbb{C}^{M \times N})}- z \mathbf{S}^* )^{-1} U  )$.
\end{proof}

\section{The Lax pair structure}\label{SecLaxP}
\noindent This section is devoted to proving theorem $\ref{LaxPairThm}$.

\subsection{The Hankel operators}
\noindent Given $d, M,N\in\mathbb{N}_+$, recall that the Hankel operators of symbol $U \in H^{\frac{1}{2}}_+(\mathbb{T}; \mathbb{C}^{M\times N})$ are given by
\begin{equation}\label{rlHankel}
\begin{split}
& \mathbf{H}^{(\mathbf{r})}_{U} : F \in L^2_+(\mathbb{T}; \mathbb{C}^{d \times N}) \mapsto \mathbf{H}^{(\mathbf{r})}_{U}(F) = \Pi_{\geq 0}(U F^*) \in L^2_+(\mathbb{T}; \mathbb{C}^{M \times d});\\
 & \mathbf{H}^{(\mathbf{l})}_{U} : G \in L^2_+(\mathbb{T}; \mathbb{C}^{M \times d}) \mapsto \mathbf{H}^{(\mathbf{l})}_{U}(G) = \Pi_{\geq 0}(G^* U) \in L^2_+(\mathbb{T}; \mathbb{C}^{d \times N}).\\
\end {split}
\end{equation}If $F=\sum_{j=1}^d \sum_{n=1}^N F_{jn}\mathbb{E}^{(dN)}_{jn} \in  L^2_+(\mathbb{T}; \mathbb{C}^{d \times N})$ and $G=\sum_{m=1}^M \sum_{k=1}^d G_{mk}\mathbb{E}^{(Md)}_{mk} \in  L^2_+(\mathbb{T}; \mathbb{C}^{M \times d})$, then
\begin{equation}\label{relHrHltoH}
\mathbf{H}^{(\mathbf{r})}_{U}(F) = \sum_{k=1}^M  \sum_{j=1}^d (\sum_{n=1}^N H_{U_{kn}}(F_{jn}))\mathbb{E}^{(Md)}_{kj}; \quad \mathbf{H}^{(\mathbf{l})}_{U}(G) = \sum_{k=1}^d  \sum_{j=1}^N (\sum_{m=1}^M H_{U_{mj}}(G_{mk}))\mathbb{E}^{(dN)}_{kj}. 
\end{equation}Both $\mathbf{H}^{(\mathbf{r})}_{U}: L^2_+(\mathbb{T}; \mathbb{C}^{d \times N}) \to L^2_+(\mathbb{T}; \mathbb{C}^{M \times d})$ and $\mathbf{H}^{(\mathbf{l})}_{U}: L^2_+(\mathbb{T}; \mathbb{C}^{M \times d}) \to  L^2_+(\mathbb{T}; \mathbb{C}^{d \times N})$ are $\mathbb{C}$-antilinear Hilbert--Schmidt operators by formula $(12)$ in page 771 of G\'erard--Grellier \cite{GGANNENS}. Precisely,  we have
\begin{equation}\label{TrHrHl}
\mathrm{Tr}(\mathbf{H}^{(\mathbf{r})}_{U} \mathbf{H}^{(\mathbf{l})}_{U}) = \mathrm{Tr}(\mathbf{H}^{(\mathbf{l})}_{U} \mathbf{H}^{(\mathbf{r})}_{U}) = \|\mathbf{H}^{(\mathbf{r})}_{U}\|_{\mathrm{HS}}^2 =\|\mathbf{H}^{(\mathbf{l})}_{U}\|_{\mathrm{HS}}^2 = d \|\sqrt{1+|\mathrm{D}|}U\|_{L^2_+(\mathbb{T}; \mathbb{C}^{M\times N})}^2.
\end{equation}In addition, for any $F \in L^2_+(\mathbb{T}; \mathbb{C}^{d \times N})$ and $G \in L^2_+(\mathbb{T}; \mathbb{C}^{M \times d})$, we have
\begin{equation}\label{HrHlinnprod}
\langle  \mathbf{H}^{(\mathbf{r})}_{U}(F), G \rangle_{L^2_+(\mathbb{T}; \mathbb{C}^{M \times d})} =\frac{1}{2\pi} \int_0^{2\pi}\mathrm{tr}\left(F(x)^* G(x)^* U(x) \right)\mathrm{d}x = \langle  \mathbf{H}^{(\mathbf{l})}_{U}(G), F \rangle_{L^2_+(\mathbb{T}; \mathbb{C}^{d \times N})}.
\end{equation}The following lemma is a direct consequence of formula \eqref{HrHlinnprod}.
\begin{lem}\label{LemofOrtho}
Given $M,N,d\in\mathbb{N}_+$ and $U \in H^{\frac{1}{2}}_+(\mathbb{T}; \mathbb{C}^{M\times N})$, we have
\begin{equation}\label{OrthoComple}
\begin{split}
& \mathrm{Ker} \mathbf{H}^{(\mathbf{l})}_{U} \mathbf{H}^{(\mathbf{r})}_{U}=\mathrm{Ker} \mathbf{H}^{(\mathbf{r})}_{U} =  (\mathrm{Im} \mathbf{H}^{(\mathbf{l})}_{U} )^{\perp} = (\mathrm{Im} \mathbf{H}^{(\mathbf{l})}_{U} \mathbf{H}^{(\mathbf{r})}_{U} )^{\perp} \subset  L^2_+(\mathbb{T}; \mathbb{C}^{d \times N}); \\
& \mathrm{Ker} \mathbf{H}^{(\mathbf{r})}_{U} \mathbf{H}^{(\mathbf{l})}_{U} =\mathrm{Ker} \mathbf{H}^{(\mathbf{l})}_{U} = (\mathrm{Im} \mathbf{H}^{(\mathbf{r})}_{U}  )^{\perp}  = (\mathrm{Im} \mathbf{H}^{(\mathbf{r})}_{U} \mathbf{H}^{(\mathbf{l})}_{U} )^{\perp}  \subset  L^2_+(\mathbb{T}; \mathbb{C}^{M \times d}).
\end{split}
\end{equation}As a consequence,  $L^2_+(\mathbb{T}; \mathbb{C}^{d \times N}) = \mathrm{Ker} \mathbf{H}^{(\mathbf{r})}_{U} \bigoplus \overline{\mathrm{Im} \mathbf{H}^{(\mathbf{l})}_{U}} = \mathrm{Ker} \mathbf{H}^{(\mathbf{l})}_{U} \mathbf{H}^{(\mathbf{r})}_{U} \bigoplus \overline{\mathrm{Im} \mathbf{H}^{(\mathbf{l})}_{U}\mathbf{H}^{(\mathbf{r})}_{U}}$ and $L^2_+(\mathbb{T}; \mathbb{C}^{M \times d}) = \mathrm{Ker} \mathbf{H}^{(\mathbf{l})}_{U} \bigoplus \overline{\mathrm{Im} \mathbf{H}^{(\mathbf{r})}_{U}} =  \mathrm{Ker} \mathbf{H}^{(\mathbf{r})}_{U} \mathbf{H}^{(\mathbf{l})}_{U} \bigoplus \overline{\mathrm{Im} \mathbf{H}^{(\mathbf{r})}_{U}\mathbf{H}^{(\mathbf{l})}_{U}}$. Furthermore, the restrictions $\mathbf{H}^{(\mathbf{r})}_{U}\big|_{\mathrm{Im} \mathbf{H}^{(\mathbf{l})}_{U}}: \mathrm{Im} \mathbf{H}^{(\mathbf{l})}_{U} \to \mathrm{Im} \mathbf{H}^{(\mathbf{r})}_{U}\mathbf{H}^{(\mathbf{l})}_{U}$ and $\mathbf{H}^{(\mathbf{l})}_{U}\big|_{\mathrm{Im} \mathbf{H}^{(\mathbf{r})}_{U}}: \mathrm{Im} \mathbf{H}^{(\mathbf{r})}_{U} \to \mathrm{Im} \mathbf{H}^{(\mathbf{l})}_{U}\mathbf{H}^{(\mathbf{r})}_{U}$ are both injective.
\end{lem}

\noindent Recall that the left and right shifted Hankel operators of the symbol $U \in H^{\frac{1}{2}}_+(\mathbb{T}; \mathbb{C}^{M\times N})$ are given by
\begin{equation}\label{rlShiftHankel}
\begin{split}
& \mathbf{K}^{(\mathbf{r})}_{U} = \mathbf{H}^{(\mathbf{r})}_{U}\mathbf{S} =   \mathbf{S}^* \mathbf{H}^{(\mathbf{r})}_{U} =\mathbf{H}^{(\mathbf{r})}_{\mathbf{S}^*  U} : F \in L^2_+(\mathbb{T}; \mathbb{C}^{d \times N}) \mapsto  \Pi_{\geq 0}\left( \mathbf{e}_{-1} UF^*\right)\in L^2_+(\mathbb{T}; \mathbb{C}^{M \times d});\\
& \mathbf{K}^{(\mathbf{l})}_{U} = \mathbf{H}^{(\mathbf{l})}_{U}\mathbf{S} =  \mathbf{S}^* \mathbf{H}^{(\mathbf{l})}_{U} =\mathbf{H}^{(\mathbf{l})}_{\mathbf{S}^*  U} : G \in L^2_+(\mathbb{T}; \mathbb{C}^{M \times d}) \mapsto  \Pi_{\geq 0}\left( \mathbf{e}_{-1} G^* U \right)\in L^2_+(\mathbb{T}; \mathbb{C}^{d \times N}).\\
\end{split}
\end{equation}We have
\begin{equation}\label{KrKlinnprod}
\langle  \mathbf{K}^{(\mathbf{r})}_{U}(F), G \rangle_{L^2_+(\mathbb{T}; \mathbb{C}^{M \times d})} =\langle  \mathbf{H}^{(\mathbf{r})}_{\mathbf{S}^* U}(F), G \rangle_{L^2_+(\mathbb{T}; \mathbb{C}^{M \times d})} = \langle  \mathbf{H}^{(\mathbf{l})}_{\mathbf{S}^* U}(G), F \rangle_{L^2_+(\mathbb{T}; \mathbb{C}^{d \times N})} = \langle  \mathbf{K}^{(\mathbf{l})}_{U}(G), F \rangle_{L^2_+(\mathbb{T}; \mathbb{C}^{d \times N})}.
\end{equation}Then $\mathrm{Ker}( \mathbf{H}^{(\mathbf{r})}_{U})$ and  $\mathrm{Ker}( \mathbf{H}^{(\mathbf{l})}_{U})$ are $\mathbf{S}$-invariant, $\mathrm{Im}( \mathbf{H}^{(\mathbf{r})}_{U})$ and  $\mathrm{Im}( \mathbf{H}^{(\mathbf{l})}_{U})$ are $\mathbf{S}^*$-invariant, i.e. 
\begin{equation}
\begin{split}
& \mathbf{S}\left(\mathrm{Ker}( \mathbf{H}^{(\mathbf{r})}_{U}) \right) \subset \mathrm{Ker}( \mathbf{H}^{(\mathbf{r})}_{U}) \subset L^2_+(\mathbb{T}; \mathbb{C}^{d \times N} ); \quad \mathbf{S}\left(\mathrm{Ker}( \mathbf{H}^{(\mathbf{l})}_{U}) \right) \subset \mathrm{Ker}( \mathbf{H}^{(\mathbf{l})}_{U}) \subset L^2_+(\mathbb{T}; \mathbb{C}^{M \times d} );  \\
& \mathbf{S}^*\left(\mathrm{Im}( \mathbf{H}^{(\mathbf{r})}_{U}) \right) \subset \mathrm{Im}( \mathbf{H}^{(\mathbf{r})}_{U}) \subset L^2_+(\mathbb{T}; \mathbb{C}^{M \times d} ); \quad \mathbf{S}^*\left(\mathrm{Im}( \mathbf{H}^{(\mathbf{l})}_{U}) \right) \subset \mathrm{Im}( \mathbf{H}^{(\mathbf{l})}_{U}) \subset L^2_+(\mathbb{T}; \mathbb{C}^{d \times N} ). 
\end{split}
\end{equation}Given any positive integers $d, M,N, P, Q\in\mathbb{N}_+$, we choose $U \in H^{\frac{1}{2}}_+(\mathbb{T}; \mathbb{C}^{M \times N})$, $V \in H^{\frac{1}{2}}_+(\mathbb{T}; \mathbb{C}^{P \times N})$ and $W \in H^{\frac{1}{2}}_+(\mathbb{T}; \mathbb{C}^{P\times Q})$. For any $x\in \mathbb{T}$, the matrices $U(x) \in \mathbb{C}^{M \times N}$ and  $V(x) \in \mathbb{C}^{P \times N}$ have the same number of columns, the $(\mathbf{rl})$-double Hankel operators are given by
\begin{equation}\label{HrHldef}
\begin{split}
\mathbf{H}^{(\mathbf{r})}_{U} \mathbf{H}^{(\mathbf{l})}_{V}: G_1 \in L^2_+(\mathbb{T}; \mathbb{C}^{P \times d}) \mapsto \mathbf{H}^{(\mathbf{r})}_{U} \mathbf{H}^{(\mathbf{l})}_{V}(G_1) \in L^2_+(\mathbb{T}; \mathbb{C}^{M \times d});\\
\mathbf{H}^{(\mathbf{r})}_{V} \mathbf{H}^{(\mathbf{l})}_{U}: G_2 \in L^2_+(\mathbb{T}; \mathbb{C}^{M \times d})  \mapsto \mathbf{H}^{(\mathbf{r})}_{V} \mathbf{H}^{(\mathbf{l})}_{U}(G_2) \in L^2_+(\mathbb{T}; \mathbb{C}^{P \times d}).
\end{split}
\end{equation}If $G_1 \in L^2_+(\mathbb{T}; \mathbb{C}^{P \times d}) $ and $G_2 \in L^2_+(\mathbb{T}; \mathbb{C}^{M \times d})$, we use \eqref{HrHlinnprod} to obtain
\begin{equation}\label{HrHladj}
\langle \mathbf{H}^{(\mathbf{r})}_{U} \mathbf{H}^{(\mathbf{l})}_{V} (G_1), G_2 \rangle_{ L^2_+(\mathbb{T}; \mathbb{C}^{M \times d})} = \langle \mathbf{H}^{(\mathbf{l})}_{U}(G_2) , \mathbf{H}^{(\mathbf{l})}_{V} (G_1) \rangle_{ L^2_+(\mathbb{T}; \mathbb{C}^{d \times N})} = \langle  G_1 , \mathbf{H}^{(\mathbf{r})}_{V} \mathbf{H}^{(\mathbf{l})}_{U}(G_2) \rangle_{ L^2_+(\mathbb{T}; \mathbb{C}^{P \times d})}.
\end{equation}If $M=P$, then $\mathbf{H}^{(\mathbf{r})}_{U} \mathbf{H}^{(\mathbf{l})}_{V}$ is a $\mathbb{C}$-linear trace class operator on $L^2_+(\mathbb{T}; \mathbb{C}^{M \times d})$ and $\mathbf{H}^{(\mathbf{r})}_{V} \mathbf{H}^{(\mathbf{l})}_{U} = \left( \mathbf{H}^{(\mathbf{r})}_{U} \mathbf{H}^{(\mathbf{l})}_{V} \right)^*$. In addition, if $U=V$, then $\mathbf{H}^{(\mathbf{r})}_{U} \mathbf{H}^{(\mathbf{l})}_{U}\geq 0$ is a $\mathbb{C}$-linear positive self-adjoint operator on $L^2_+(\mathbb{T}; \mathbb{C}^{M \times d})$ of trace class.\\

\noindent For any $x\in \mathbb{T}$, the matrices $V(x) \in \mathbb{C}^{P \times N}$ and $W(x) \in \mathbb{C}^{P \times Q}$ have the same number of rows, the $(\mathbf{lr})$-double Hankel operators are given by
\begin{equation}\label{HlHrdef}
\begin{split}
\mathbf{H}^{(\mathbf{l})}_{V} \mathbf{H}^{(\mathbf{r})}_{W} : F_1 \in L^2_+(\mathbb{T}; \mathbb{C}^{d \times Q}) \mapsto \mathbf{H}^{(\mathbf{l})}_{V} \mathbf{H}^{(\mathbf{r})}_{W}(F_1) \in L^2_+(\mathbb{T}; \mathbb{C}^{d \times N});\\
\mathbf{H}^{(\mathbf{l})}_{W} \mathbf{H}^{(\mathbf{r})}_{V} : F_2 \in L^2_+(\mathbb{T}; \mathbb{C}^{d \times N}) \mapsto \mathbf{H}^{(\mathbf{l})}_{W} \mathbf{H}^{(\mathbf{r})}_{V}(F_2) \in L^2_+(\mathbb{T}; \mathbb{C}^{d \times Q}).
\end{split}
\end{equation}If $F_1 \in L^2_+(\mathbb{T}; \mathbb{C}^{d \times Q})$ and $F_2 \in L^2_+(\mathbb{T}; \mathbb{C}^{d \times N})$, we use \eqref{HrHlinnprod} to obtain 
\begin{equation}\label{HlHradj}
\langle \mathbf{H}^{(\mathbf{l})}_{V} \mathbf{H}^{(\mathbf{r})}_{W} (F_1), F_2 \rangle_{ L^2_+(\mathbb{T}; \mathbb{C}^{d \times N})} = \langle \mathbf{H}^{(\mathbf{r})}_{V} (F_2) ,  \mathbf{H}^{(\mathbf{r})}_{W} (F_1) \rangle_{ L^2_+(\mathbb{T}; \mathbb{C}^{P \times d})} = \langle  F_1 , \mathbf{H}^{(\mathbf{l})}_{W} \mathbf{H}^{(\mathbf{r})}_{V} (F_2) \rangle_{ L^2_+(\mathbb{T}; \mathbb{C}^{d \times Q})}.
\end{equation}If $N=Q$, then $\mathbf{H}^{(\mathbf{l})}_{V} \mathbf{H}^{(\mathbf{r})}_{W}$ is a $\mathbb{C}$-linear trace class operator on $L^2_+(\mathbb{T}; \mathbb{C}^{d \times Q})$ and $\mathbf{H}^{(\mathbf{l})}_{W} \mathbf{H}^{(\mathbf{r})}_{V} = \left( \mathbf{H}^{(\mathbf{l})}_{V} \mathbf{H}^{(\mathbf{r})}_{W} \right)^*$. In addition, if $V=W$, then $\mathbf{H}^{(\mathbf{l})}_{W} \mathbf{H}^{(\mathbf{r})}_{W}\geq 0$ is a $\mathbb{C}$-linear positive self-adjoint operator on $L^2_+(\mathbb{T}; \mathbb{C}^{d \times Q})$ of trace class. \\

\noindent Since  $\mathbf{S}^* U \in H^{\frac{1}{2}}_+(\mathbb{T}; \mathbb{C}^{M \times N})$, $\mathbf{S}^* V \in H^{\frac{1}{2}}_+(\mathbb{T}; \mathbb{C}^{P \times N})$ and $\mathbf{S}^* W \in H^{\frac{1}{2}}_+(\mathbb{T}; \mathbb{C}^{P\times Q})$, the double shifted Hankel operators are given by
\begin{equation}
\begin{split}
& \mathbf{K}^{(\mathbf{r})}_{U} \mathbf{K}^{(\mathbf{l})}_{V} = \mathbf{H}^{(\mathbf{r})}_{\mathbf{S}^* U} \mathbf{H}^{(\mathbf{l})}_{\mathbf{S}^* V} ; \qquad \mathbf{K}^{(\mathbf{r})}_{V} \mathbf{K}^{(\mathbf{l})}_{U} = \mathbf{H}^{(\mathbf{r})}_{\mathbf{S}^* V} \mathbf{H}^{(\mathbf{l})}_{\mathbf{S}^* U}; \\
& \mathbf{K}^{(\mathbf{l})}_{V} \mathbf{K}^{(\mathbf{r})}_{W}=\mathbf{H}^{(\mathbf{l})}_{\mathbf{S}^* V} \mathbf{H}^{(\mathbf{r})}_{\mathbf{S}^*  W}; \qquad \mathbf{K}^{(\mathbf{l})}_{W} \mathbf{K}^{(\mathbf{r})}_{V}=\mathbf{H}^{(\mathbf{l})}_{\mathbf{S}^* W} \mathbf{H}^{(\mathbf{r})}_{\mathbf{S}^*  V}.
\end{split}
\end{equation}If $G_1 \in L^2_+(\mathbb{T}; \mathbb{C}^{P \times d}) $, $G_2 \in L^2_+(\mathbb{T}; \mathbb{C}^{M \times d})$, $F_1 \in L^2_+(\mathbb{T}; \mathbb{C}^{d \times Q})$ and $F_2 \in L^2_+(\mathbb{T}; \mathbb{C}^{d \times N})$, formulas \eqref{HrHladj} and \eqref{HlHradj} yield that
\begin{equation}
\begin{split}
& \langle \mathbf{K}^{(\mathbf{r})}_{U} \mathbf{K}^{(\mathbf{l})}_{V} (G_1), G_2 \rangle_{ L^2_+(\mathbb{T}; \mathbb{C}^{M \times d})} = \langle \mathbf{K}^{(\mathbf{l})}_{U}(G_2) , \mathbf{K}^{(\mathbf{l})}_{V} (G_1) \rangle_{ L^2_+(\mathbb{T}; \mathbb{C}^{d \times N})} = \langle  G_1 , \mathbf{K}^{(\mathbf{r})}_{V} \mathbf{K}^{(\mathbf{l})}_{U}(G_2) \rangle_{ L^2_+(\mathbb{T}; \mathbb{C}^{P \times d})};\\
& \langle \mathbf{K}^{(\mathbf{l})}_{V} \mathbf{K}^{(\mathbf{r})}_{W} (F_1), F_2 \rangle_{ L^2_+(\mathbb{T}; \mathbb{C}^{d \times N})} = \langle \mathbf{K}^{(\mathbf{r})}_{V} (F_2) ,  \mathbf{K}^{(\mathbf{r})}_{W} (F_1) \rangle_{ L^2_+(\mathbb{T}; \mathbb{C}^{P \times d})} = \langle  F_1 , \mathbf{K}^{(\mathbf{l})}_{W} \mathbf{K}^{(\mathbf{r})}_{V} (F_2) \rangle_{ L^2_+(\mathbb{T}; \mathbb{C}^{d \times Q})}.
\end{split}
\end{equation}If $M=P$, then $\mathbf{K}^{(\mathbf{r})}_{U} \mathbf{K}^{(\mathbf{l})}_{V}$ is a $\mathbb{C}$-linear trace class operator on $L^2_+(\mathbb{T}; \mathbb{C}^{M \times d})$ and $\mathbf{K}^{(\mathbf{r})}_{V} \mathbf{K}^{(\mathbf{l})}_{U} = \left( \mathbf{K}^{(\mathbf{r})}_{U} \mathbf{K}^{(\mathbf{l})}_{V} \right)^*$. In addition, if $U=V$, then $\mathbf{K}^{(\mathbf{r})}_{U} \mathbf{K}^{(\mathbf{l})}_{U}\geq 0$ is a $\mathbb{C}$-linear positive operator on $L^2_+(\mathbb{T}; \mathbb{C}^{M \times d})$ of trace class.\\

\noindent If $N=Q$, then $\mathbf{K}^{(\mathbf{l})}_{V} \mathbf{K}^{(\mathbf{r})}_{W}$ is a $\mathbb{C}$-linear trace class operator on $L^2_+(\mathbb{T}; \mathbb{C}^{d \times Q})$ and $\mathbf{K}^{(\mathbf{l})}_{W} \mathbf{K}^{(\mathbf{r})}_{V} = \left( \mathbf{K}^{(\mathbf{l})}_{V} \mathbf{K}^{(\mathbf{r})}_{W} \right)^*$. In addition, if $V=W$, then $\mathbf{K}^{(\mathbf{l})}_{W} \mathbf{K}^{(\mathbf{r})}_{W}\geq 0$ is a $\mathbb{C}$-linear positive operator on $L^2_+(\mathbb{T}; \mathbb{C}^{d \times Q})$ of trace class.\\

\begin{lem}\label{2H2Klemma}
Given  $M,N, P, Q\in\mathbb{N}_+$,  $U \in H^{\frac{1}{2}}_+(\mathbb{T}; \mathbb{C}^{M \times N})$, $V \in H^{\frac{1}{2}}_+(\mathbb{T}; \mathbb{C}^{P \times N})$ and $W \in H^{\frac{1}{2}}_+(\mathbb{T}; \mathbb{C}^{P\times Q})$, if $G  \in L^2_+(\mathbb{T}; \mathbb{C}^{P \times d})$, $F \in  L^2_+(\mathbb{T}; \mathbb{C}^{d \times Q})$ for some $d \in \mathbb{N}_+$, then we have
\begin{equation}\label{rel2K2Hfor}
\mathbf{K}^{(\mathbf{r})}_{U} \mathbf{K}^{(\mathbf{l})}_{V}(G) = \mathbf{H}^{(\mathbf{r})}_{U} \mathbf{H}^{(\mathbf{l})}_{V}(G) - U\widehat{V^* G}(0), \quad \mathbf{K}^{(\mathbf{l})}_{V} \mathbf{K}^{(\mathbf{r})}_{W}(F) = \mathbf{H}^{(\mathbf{l})}_{V} \mathbf{H}^{(\mathbf{r})}_{W}(F) -  \widehat{F W^*}(0)V.
\end{equation}
\end{lem}
\begin{proof}
Formula \eqref{rlShiftHankel} yields that $\mathbf{K}^{(\mathbf{r})}_{U} \mathbf{K}^{(\mathbf{l})}_{V} = \mathbf{H}^{(\mathbf{r})}_{U} \mathbf{S}  \mathbf{S}^* \mathbf{H}^{(\mathbf{l})}_{V}$ and $\mathbf{K}^{(\mathbf{l})}_{V} \mathbf{K}^{(\mathbf{r})}_{W} = \mathbf{H}^{(\mathbf{l})}_{V} \mathbf{S}  \mathbf{S}^* \mathbf{H}^{(\mathbf{r})}_{W}$. Moreover, we have $\mathbf{H}^{(\mathbf{r})}_{U} \left(\widehat{G^* V}(0) \right)=U\widehat{V^* G}(0) \in  L^2_+(\mathbb{T}; \mathbb{C}^{M \times d})$ and $\mathbf{H}^{(\mathbf{l})}_{V} \left( \widehat{WF^*}(0) \right)= \widehat{F W^*}(0)V \in  L^2_+(\mathbb{T}; \mathbb{C}^{d \times N})$. It suffices to conclude by formula \eqref{inverseSS*}.
\end{proof}
\begin{lem}\label{InvSpaceeHrHl}
Given $M,N,d\in\mathbb{N}_+$ and $t \in \mathbb{R}$, if $U \in H^{\frac{1}{2}}_+(\mathbb{T}; \mathbb{C}^{M \times N})$, then
\begin{equation}\label{ImHuInv}
\begin{split}
& e^{it \mathbf{K}^{(\mathbf{r})}_{U} \mathbf{K}^{(\mathbf{l})}_{U}}\mathbf{S}^* e^{-it \mathbf{H}^{(\mathbf{r})}_{U} \mathbf{H}^{(\mathbf{l})}_{U}}  (\mathrm{Im} \mathbf{H}^{(\mathbf{r})}_{U} ) \subset  \mathrm{Im} \mathbf{K}^{(\mathbf{r})}_{U} \subset \mathrm{Im} \mathbf{H}^{(\mathbf{r})}_{U} \subset L^{2}_+(\mathbb{T}; \mathbb{C}^{M \times d});\\
& e^{it \mathbf{K}^{(\mathbf{l})}_{U} \mathbf{K}^{(\mathbf{r})}_{U}}\mathbf{S}^* e^{-it \mathbf{H}^{(\mathbf{l})}_{U} \mathbf{H}^{(\mathbf{r})}_{U}} ( \mathrm{Im} \mathbf{H}^{(\mathbf{l})}_{U} ) \subset  \mathrm{Im} \mathbf{K}^{(\mathbf{l})}_{U} \subset \mathrm{Im} \mathbf{H}^{(\mathbf{l})}_{U} \subset L^{2}_+(\mathbb{T}; \mathbb{C}^{d \times N}).
\end{split}
\end{equation}
\end{lem}
\begin{proof}
Since $\left(\mathbf{H}^{(\mathbf{r})}_{U} \mathbf{H}^{(\mathbf{l})}_{U}\right)^n \mathbf{H}^{(\mathbf{r})}_{U} =\mathbf{H}^{(\mathbf{r})}_{U} \left(\mathbf{H}^{(\mathbf{l})}_{U} \mathbf{H}^{(\mathbf{r})}_{U}\right)^n$ and $\left(\mathbf{H}^{(\mathbf{l})}_{U} \mathbf{H}^{(\mathbf{r})}_{U}\right)^n \mathbf{H}^{(\mathbf{l})}_{U} =\mathbf{H}^{(\mathbf{l})}_{U} \left(\mathbf{H}^{(\mathbf{r})}_{U} \mathbf{H}^{(\mathbf{l})}_{U}\right)^n$, $\forall n \in \mathbb{N}$, the power series of $\exp$ in $\mathcal{B}(L^{2}_+(\mathbb{T}; \mathbb{C}^{M \times d}))$ and $\mathcal{B}(L^{2}_+(\mathbb{T}; \mathbb{C}^{d \times N}))$ yields that
\begin{equation}\label{HExpHH}
\begin{split}
& e^{-it \mathbf{H}^{(\mathbf{r})}_{U} \mathbf{H}^{(\mathbf{l})}_{U}}\mathbf{H}^{(\mathbf{r})}_{U} = \mathbf{H}^{(\mathbf{r})}_{U} e^{it \mathbf{H}^{(\mathbf{l})}_{U} \mathbf{H}^{(\mathbf{r})}_{U}}; \quad e^{-it \mathbf{H}^{(\mathbf{l})}_{U} \mathbf{H}^{(\mathbf{r})}_{U}}\mathbf{H}^{(\mathbf{l})}_{U} = \mathbf{H}^{(\mathbf{l})}_{U} e^{it \mathbf{H}^{(\mathbf{r})}_{U} \mathbf{H}^{(\mathbf{l})}_{U}}; \\
& e^{it \mathbf{K}^{(\mathbf{r})}_{U} \mathbf{K}^{(\mathbf{l})}_{U}}\mathbf{K}^{(\mathbf{r})}_{U} = \mathbf{K}^{(\mathbf{r})}_{U} e^{-it \mathbf{K}^{(\mathbf{l})}_{U} \mathbf{K}^{(\mathbf{r})}_{U}}; \quad  e^{it \mathbf{K}^{(\mathbf{l})}_{U} \mathbf{K}^{(\mathbf{r})}_{U}}\mathbf{K}^{(\mathbf{l})}_{U} = \mathbf{K}^{(\mathbf{l})}_{U} e^{-it \mathbf{K}^{(\mathbf{r})}_{U} \mathbf{K}^{(\mathbf{l})}_{U}}. 
 \end{split}
\end{equation}It suffices to conclude by \eqref{rlShiftHankel}.
\end{proof}

\subsection{The Kronecker theorem}
\begin{defi}
Given a positive integer $n \in\mathbb{N}_+$, let $\mathcal{M}(n)$ denote the set of rational functions $u = \frac{p(\mathbf{e}_1)}{q(\mathbf{e}_1)}$ such that $p \in \mathbb{C}_{\leq n-1}[X], q\in \mathbb{C}_{\leq n}[X]$, the polynomials $p$ and $q$ have no common divisors, $q(0)=1$, $q^{-1}\{0\} \subset \mathbb{C} \backslash \overline{D}(0,1)$, $\deg p= n-1$ or  $\deg q= n$. We set $\mathcal{M}(0)=\{0\}$ and $\mathcal{M}_{\mathrm{FR}}:=\bigcup_{n\in \mathbb{N}}\mathcal{M}(n)$.
\end{defi}
\noindent If $u \in L^2_+(\mathbb{T}; \mathbb{C})$, the Kronecker theorem \cite{Kronecker1881} yields the following equivalence: $\forall n \in \mathbb{N}$, 
\begin{equation}\label{1*1Kroneckerthm}
u \in \mathcal{M}(n) \Longleftrightarrow \mathrm{r}(H_u) = \dim_{\mathbb{C}} \mathrm{Im}H_u =n.
\end{equation}We refer to Appendix 4 (subsection 10.4) of G\'erard--Grellier \cite{GGANNENS} for the proof of \eqref{1*1Kroneckerthm}. Given $M,N \in \mathbb{N}_+$,  
\begin{equation}
\mathcal{M}_{\mathrm{FR}}^{M \times N} = \{\tfrac{A(\mathbf{e}_1)}{q(\mathbf{e}_1)}: q\in \mathbb{C}[X], \; q^{-1}\{0\} \subset \mathbb{C} \backslash \overline{D}(0,1), \; A\in  (\mathbb{C}[X])^{M \times N}\}\subset C^{\infty}_+(\mathbb{T}; \mathbb{C}^{M \times N}).
\end{equation}
\begin{prop}\label{WKroneckerM*N}
Given $U \in L^2_+(\mathbb{T}; \mathbb{C}^{M \times N})$ for some $M,N \in \mathbb{N}$, then each of the following three properties implies the others: \\

\noindent $\mathrm{(a)}$. $U \in \mathcal{M}_{\mathrm{FR}}^{M \times N}$.\\
\noindent $\mathrm{(b)}$. Both $\mathbf{H}^{(\mathbf{r})}_{U}: L^2_+(\mathbb{T}; \mathbb{C}^{d \times N}) \to L^2_+(\mathbb{T}; \mathbb{C}^{M \times d})$ and $\mathbf{H}^{(\mathbf{l})}_{U}: L^2_+(\mathbb{T}; \mathbb{C}^{M \times d}) \to L^2_+(\mathbb{T}; \mathbb{C}^{d \times N})$ are finite-rank operators, $\forall d \in \mathbb{N}_+$, and $\dim_{\mathbb{C}} \mathrm{Im}\mathbf{H}^{(\mathbf{r})}_{U} = \dim_{\mathbb{C}} \mathrm{Im}\mathbf{H}^{(\mathbf{l})}_{U} = \dim_{\mathbb{C}} \mathrm{Im}\mathbf{H}^{(\mathbf{r})}_{U}\mathbf{H}^{(\mathbf{l})}_{U} = \dim_{\mathbb{C}}\mathrm{Im}\mathbf{H}^{(\mathbf{l})}_{U}\mathbf{H}^{(\mathbf{r})}_{U}<+\infty$. \\
\noindent $\mathrm{(c)}$. There exists $d  \in \mathbb{N}_+$ such that at least one of the subspaces $\mathrm{Im}\mathbf{H}^{(\mathbf{r})}_{U}$, $\mathrm{Im}\mathbf{H}^{(\mathbf{l})}_{U}$, $ \mathrm{Im}\mathbf{H}^{(\mathbf{r})}_{U}\mathbf{H}^{(\mathbf{l})}_{U}$, $\mathrm{Im}\mathbf{H}^{(\mathbf{l})}_{U}\mathbf{H}^{(\mathbf{r})}_{U}$ has finite dimension. 
\end{prop}
\begin{proof}
$\mathrm{(a)} \Rightarrow \mathrm{(b)}$: If $U =\sum_{k=1}^M\sum_{n=1}^N U_{kn}\mathbb{E}_{kn}^{(MN)}\in \mathcal{M}_{\mathrm{FR}}^{M \times N}$, then $\mathbb{V}_U:=\sum_{k=1}^M\sum_{n=1}^N \mathrm{Im}H_{U_{kn}}$ is a finite dimensional subspace of $L^2_+(\mathbb{T}; \mathbb{C})$ by \eqref{1*1Kroneckerthm}. For any $d \in \mathbb{N}_+$, formula \eqref{relHrHltoH} yields that $\mathbf{H}^{(\mathbf{r})}_{U} \subset \mathbb{V}_U^{M \times d}$. If one of the subspaces $\mathrm{Im}\mathbf{H}^{(\mathbf{r})}_{U}$, $\mathrm{Im}\mathbf{H}^{(\mathbf{l})}_{U}$, $ \mathrm{Im}\mathbf{H}^{(\mathbf{r})}_{U}\mathbf{H}^{(\mathbf{l})}_{U}$, $\mathrm{Im}\mathbf{H}^{(\mathbf{l})}_{U}\mathbf{H}^{(\mathbf{r})}_{U}$ has finite dimension, then Lemma $\ref{LemofOrtho}$ implies that $\dim_{\mathbb{C}} \mathrm{Im}\mathbf{H}^{(\mathbf{r})}_{U} = \dim_{\mathbb{C}} \mathrm{Im}\mathbf{H}^{(\mathbf{l})}_{U} = \dim_{\mathbb{C}} \mathrm{Im}\mathbf{H}^{(\mathbf{r})}_{U}\mathbf{H}^{(\mathbf{l})}_{U} = \dim_{\mathbb{C}}\mathrm{Im}\mathbf{H}^{(\mathbf{l})}_{U}\mathbf{H}^{(\mathbf{r})}_{U}<+\infty$. $\mathrm{(c)} \Rightarrow \mathrm{(a)}$: If $U \in L^2_+(\mathbb{T}; \mathbb{C}^{M \times N})$  such that $\mathrm{(c)}$ holds, then $\mathbf{H}^{(\mathbf{r})}_{U} \in \mathrm{HS}(L^2_+(\mathbb{T}; \mathbb{C}^{d \times N});  L^2_+(\mathbb{T}; \mathbb{C}^{M \times d}))$ and $U \in H^{\frac{1}{2}}_+(\mathbb{T}; \mathbb{C}^{M \times N})$ by  \eqref{TrHrHl}. Assume that $U \in H^{\frac{1}{2}}_+(\mathbb{T}; \mathbb{C}^{M \times N}) \backslash \mathcal{M}_{\mathrm{FR}}^{M \times N}$, then $H_{U_{st}}$ has infinite rank for some $1 \leq s \leq M$ and $1\leq t \leq N$, thanks to \eqref{1*1Kroneckerthm}. For any $R \in \mathbb{N}_+$, there exists $f_1, f_2, \cdots, f_R\in L^2_+(\mathbb{T};\mathbb{C})$  such that $\{H_{U_{st}}(f_l)\}_{1\leq l \leq R}$ is linearly independent in $L^2_+(\mathbb{T};\mathbb{C})$. Then $\{\mathbf{H}^{(\mathbf{r})}_{U}(f_l \mathbb{E}_{1t}^{(dN)})\}_{1\leq l \leq R}$ is linearly independent in $\mathrm{Im}\mathbf{H}^{(\mathbf{r})}_{U}  \subset  L^2_+(\mathbb{T}; \mathbb{C}^{M \times d})$. So $\dim_{\mathbb{C}} \mathrm{Im}\mathbf{H}^{(\mathbf{r})}_{U}=+\infty$, $\forall d \in \mathbb{N}_+$, which contradicts $\mathrm{(c)}$.
\end{proof}
\begin{rem}\label{ImHr=ImHrHl}
Given $M,N \in \mathbb{N}_+$, if $U \in \mathcal{M}_{\mathrm{FR}}^{M \times N}$, Proposition $\ref{WKroneckerM*N}$ yields that $\mathrm{Im}\mathbf{H}^{(\mathbf{r})}_{U} = \mathrm{Im}\mathbf{H}^{(\mathbf{r})}_{U}\mathbf{H}^{(\mathbf{l})}_{U} = \mathbf{H}^{(\mathbf{r})}_{U} \mathrm{Im}\mathbf{H}^{(\mathbf{l})}_{U} \subset L^2_+(\mathbb{T}; \mathbb{C}^{M \times d})$ and $\mathrm{Im}\mathbf{H}^{(\mathbf{l})}_{U} = \mathrm{Im}\mathbf{H}^{(\mathbf{l})}_{U}\mathbf{H}^{(\mathbf{r})}_{U} = \mathbf{H}^{(\mathbf{l})}_{U} \mathrm{Im}\mathbf{H}^{(\mathbf{r})}_{U} \subset L^2_+(\mathbb{T}; \mathbb{C}^{d \times N})$,  $\forall d \in \mathbb{N}_+$.
\end{rem}

\begin{lem}\label{denseFR}
Given $M,N \in \mathbb{N}_+$ and $s \geq 0$, the set $(\mathcal{M}_{\mathrm{FR}}\backslash \{0\})^{M \times N}$ is dense in $H^s_+(\mathbb{T}; \mathbb{C}^{M \times N})$.
\end{lem}
\begin{proof}
If $U = \sum_{n \geq 0} \hat{U}(n)\mathbf{e}_n \in H^s_+(\mathbb{T}; \mathbb{C}^{M \times N})$, set $V^{(m)}:=\sum_{n = 0}^m \hat{U}(n)\mathbf{e}_n = \sum_{k = 1}^M \sum_{j = 1}^N V^{(m)}_{kj} \mathbb{E}_{kj}^{(MN)}$, $\boldsymbol{\Lambda}_m:=\{(k,j) :V^{(m)}_{kj} =0 \}$ and $\tilde{V}^{(m)}:=V^{(m)} + 2^{-m}\sum_{(k,j) \in \boldsymbol{\Lambda}_m}\mathbb{E}_{kj}^{(MN)} \in (\mathcal{M}_{\mathrm{FR}}\backslash \{0\})^{M \times N}$, $\forall m \in \mathbb{N}$. Then $\tilde{V}^{(m)} \to U$ in $H^s_+(\mathbb{T}; \mathbb{C}^{M \times N})$, as $m \to +\infty$.
\end{proof}

\subsection{The Toeplitz operators}
\noindent Given $d, M,N\in\mathbb{N}_+$, recall that the Toeplitz operators of symbol $V \in L^2(\mathbb{T}; \mathbb{C}^{M \times N})$ are given by 
\begin{equation}\label{rlToep}
\begin{split}
& \mathbf{T}^{(\mathbf{r})}_{V} : G \in H^1_+(\mathbb{T}; \mathbb{C}^{N \times d}) \mapsto \mathbf{T}^{(\mathbf{r})}_{V}(G) = \Pi_{\geq 0}(V G) \in L^2_+(\mathbb{T}; \mathbb{C}^{M \times d}),\\
& \mathbf{T}^{(\mathbf{l})}_{V} : F \in H^1_+(\mathbb{T}; \mathbb{C}^{d \times M}) \mapsto \mathbf{T}^{(\mathbf{l})}_{V}(F) = \Pi_{\geq 0}(FV) \in L^2_+(\mathbb{T}; \mathbb{C}^{d \times N}).\\
\end {split}
\end{equation}If $V \in L^{\infty}(\mathbb{T}; \mathbb{C}^{M \times N})$, then $\mathbf{T}^{(\mathbf{r})}_{V}:  L^2_+(\mathbb{T}; \mathbb{C}^{N \times d}) \to L^2_+(\mathbb{T}; \mathbb{C}^{M \times d})$ and $\mathbf{T}^{(\mathbf{l})}_{V}: L^2_+(\mathbb{T}; \mathbb{C}^{d \times M}) \to L^2_+(\mathbb{T}; \mathbb{C}^{d \times N})$ are both bounded operators. Moreover, $\forall G \in L^2_+(\mathbb{T}; \mathbb{C}^{N \times d}), \forall A \in L^2_+(\mathbb{T}; \mathbb{C}^{M \times d})$, we have
\begin{equation}\label{<TrG,A>}
\langle \mathbf{T}^{(\mathbf{r})}_{V}(G), A\rangle_{L^2_+(\mathbb{T}; \mathbb{C}^{M \times d})} = \langle G, \mathbf{T}^{(\mathbf{r})}_{V^*}(A)\rangle_{L^2_+(\mathbb{T}; \mathbb{C}^{N \times d})}. 
\end{equation}If $V \in L^{\infty}(\mathbb{T}; \mathbb{C}^{M \times N})$, $\forall F \in L^2_+(\mathbb{T}; \mathbb{C}^{d \times M}), \forall B \in L^2_+(\mathbb{T}; \mathbb{C}^{d \times N})$, we have
\begin{equation}\label{<TlF,B>}
\langle \mathbf{T}^{(\mathbf{l})}_{V}(F), B\rangle_{L^2_+(\mathbb{T}; \mathbb{C}^{d \times N})} = \langle F, \mathbf{T}^{(\mathbf{l})}_{V^*}(B)\rangle_{L^2_+(\mathbb{T}; \mathbb{C}^{d \times M})}. 
\end{equation}Set $M=N$. If $V \in L^{\infty}(\mathbb{T}; \mathbb{C}^{N \times N})$, then \eqref{<TrG,A>} and \eqref{<TlF,B>} imply that $\mathbf{T}^{(\mathbf{r})}_{V^*} = (\mathbf{T}^{(\mathbf{r})}_{V})^* \in \mathcal{B}(L^2_+(\mathbb{T}; \mathbb{C}^{N \times d}))$ and $\mathbf{T}^{(\mathbf{l})}_{V^*} = (\mathbf{T}^{(\mathbf{l})}_{V})^* \in \mathcal{B}(L^2_+(\mathbb{T}; \mathbb{C}^{d \times N}))$. The next lemma shows some commutator formulas between the Toeplitz operators and shift operators.
\begin{lem}\label{CommTSTS*lem}
Given $d, M,N\in \mathbb{N}_+$, if $B \in L^{\infty}(\mathbb{T}; \mathbb{C}^{M\times N})$, $G \in L^2_+(\mathbb{T}; \mathbb{C}^{N \times d})$, $F\in L^2_+(\mathbb{T}; \mathbb{C}^{d \times M})$, then
\begin{equation}\label{ForCommTSS*}
\begin{split}
& \left[\mathbf{T}^{(\mathbf{r})}_{B}, \mathbf{S}\right](G) = \widehat{BG}(-1) \in \mathbb{C}^{M \times d}; \quad \left[\mathbf{S}^*, \mathbf{T}^{(\mathbf{r})}_{B}\right](G) = \mathbf{S}^*\left(\Pi_{\geq 0}B \right)\hat{G}(0) \in  L^2_+(\mathbb{T}; \mathbb{C}^{M \times d});\\
& \left[\mathbf{T}^{(\mathbf{l})}_{B}, \mathbf{S}\right](F) = \widehat{FB}(-1) \in \mathbb{C}^{d \times N}; \quad \left[\mathbf{S}^*, \mathbf{T}^{(\mathbf{l})}_{B}\right](F) = \hat{F}(0) \mathbf{S}^*\left(\Pi_{\geq 0}B \right) \in  L^2_+(\mathbb{T}; \mathbb{C}^{d \times N}).\\
\end{split}
\end{equation}
\end{lem}
\begin{proof}
Since $\Pi_{\geq 0} \left(\mathbf{e}_1 \Pi_{<0} (BG) \right) = \widehat{BG}(-1)$ and $\Pi_{\geq 0} \left(\mathbf{e}_1 \Pi_{<0} (FB) \right) = \widehat{FB}(-1)$. So we have $\mathbf{T}^{(\mathbf{r})}_{B} \mathbf{S} (G) = \mathbf{e}_1 \mathbf{T}^{(\mathbf{r})}_{B}(G)  + \Pi_{\geq 0}\left(\mathbf{e}_1 \Pi_{< 0}(BG)\right)= \mathbf{S}\mathbf{T}^{(\mathbf{r})}_{B}(G) +  \widehat{BG}(-1)$ and $\mathbf{T}^{(\mathbf{l})}_{B} \mathbf{S} (F) = \mathbf{e}_1 \mathbf{T}^{(\mathbf{l})}_{B}(F)  + \Pi_{\geq 0}\left(\mathbf{e}_1 \Pi_{< 0}(FB)\right)= \mathbf{S}\mathbf{T}^{(\mathbf{l})}_{B}(F) +  \widehat{FB}(-1)$. Since $\mathbf{e}_{-1} \Pi_{<0}(BG) \in L^2_-(\mathbb{T}; \mathbb{C}^{M \times d})$ and $\Pi_{\geq 0}\left(B \Pi_{<0}(\mathbf{e}_{-1}G) \right)=\Pi_{\geq 0}\left(\mathbf{e}_{-1} B \right)\hat{G}(0)$, so $\mathbf{S}^* \mathbf{T}^{(\mathbf{r})}_{B} (G) = \Pi_{\geq 0} \left(\mathbf{e}_{-1} B G \right)= \mathbf{T}^{(\mathbf{r})}_{B}\mathbf{S}^* (G)+\Pi_{\geq 0} \left(\mathbf{e}_{-1} (\Pi_{\geq 0} B)  \right)\hat{G}(0) = \mathbf{T}^{(\mathbf{r})}_{B}\mathbf{S}^* (G) + \mathbf{S}^*\left(\Pi_{\geq 0}B \right)\hat{G}(0)$. Since $\mathbf{e}_{-1}\Pi_{<0}(FB) \in L^2(\mathbb{T}; \mathbb{C}^{d \times N})$ and $\Pi_{\geq 0}\left(\Pi_{<0}(\mathbf{e}_{-1}F) B \right)=\hat{F}(0) \Pi_{\geq 0}\left(\mathbf{e}_{-1} B \right)=\hat{F}(0) \Pi_{\geq 0}\left(\mathbf{e}_{-1} (\Pi_{\geq 0}B) \right)$, we have $\mathbf{S}^* \mathbf{T}^{(\mathbf{l})}_{B} (F) =\Pi_{\geq 0}\left( \mathbf{e}_{-1} F B  \right) = \mathbf{T}^{(\mathbf{l})}_{B}  \mathbf{S}^* (F) + \hat{F}(0) \mathbf{S}^*\left(\Pi_{\geq 0}B \right)$. 
\end{proof}
 
\noindent Given any positive integers $d, M,N, P, Q\in\mathbb{N}_+$, we choose $A \in L^{\infty} (\mathbb{T}; \mathbb{C}^{M \times N})$, $B \in   L^{\infty}(\mathbb{T}; \mathbb{C}^{N \times P})$ and $C \in  L^{\infty}(\mathbb{T}; \mathbb{C}^{P\times Q})$. The following double Toeplitz operators are bounded:
\begin{equation}\label{2Toep}
\begin{split}
&\mathbf{T}^{(\mathbf{r})}_{A} \mathbf{T}^{(\mathbf{r})}_{B} : G \in  L^2_+(\mathbb{T}; \mathbb{C}^{P  \times d}) \mapsto \mathbf{T}^{(\mathbf{r})}_{A} \mathbf{T}^{(\mathbf{r})}_{B}(G) = \Pi_{\geq 0}\left( A\Pi_{\geq 0}(BG)\right)\in  L^2_+(\mathbb{T}; \mathbb{C}^{M \times d});\\
& \mathbf{T}^{(\mathbf{r})}_{AB} : G \in  L^2_+(\mathbb{T}; \mathbb{C}^{P  \times d}) \mapsto \mathbf{T}^{(\mathbf{r})}_{AB}(G) = \Pi_{\geq 0}\left(ABG\right)\in  L^2_+(\mathbb{T}; \mathbb{C}^{M \times d});\\
&\mathbf{T}^{(\mathbf{l})}_{C} \mathbf{T}^{(\mathbf{l})}_{B} : F \in  L^2_+(\mathbb{T}; \mathbb{C}^{d  \times N}) \mapsto \mathbf{T}^{(\mathbf{l})}_{C} \mathbf{T}^{(\mathbf{l})}_{B}(G) = \Pi_{\geq 0}\left( \Pi_{\geq 0}(FB)C\right)\in  L^2_+(\mathbb{T}; \mathbb{C}^{d\times Q});\\
& \mathbf{T}^{(\mathbf{l})}_{BC} : F \in  L^2_+(\mathbb{T}; \mathbb{C}^{d  \times N}) \mapsto \mathbf{T}^{(\mathbf{l})}_{BC}(G) = \Pi_{\geq 0}\left(FBC\right)\in  L^2_+(\mathbb{T}; \mathbb{C}^{d\times Q}).
\end{split}
\end{equation}

\begin{lem}\label{2K2Tlemma}
Given  $M,N, P, Q\in\mathbb{N}_+$,  $U \in H^{\frac{1}{2}}_+(\mathbb{T}; \mathbb{C}^{M \times N})$, $V \in H^{\frac{1}{2}}_+(\mathbb{T}; \mathbb{C}^{P \times N})$ and $W \in H^{\frac{1}{2}}_+(\mathbb{T}; \mathbb{C}^{P\times Q})$, then we have
\begin{equation}\label{rel2K&T-TTfor}
\mathbf{K}^{(\mathbf{r})}_{U} \mathbf{K}^{(\mathbf{l})}_{V}  = \mathbf{T}^{(\mathbf{r})}_{U V^*}  - \mathbf{T}^{(\mathbf{r})}_{U}\mathbf{T}^{(\mathbf{r})}_{V^*}, \quad \mathbf{K}^{(\mathbf{l})}_{V} \mathbf{K}^{(\mathbf{r})}_{W} = \mathbf{T}^{(\mathbf{l})}_{W^* V}  - \mathbf{T}^{(\mathbf{l})}_{V}\mathbf{T}^{(\mathbf{l})}_{W^*}.
\end{equation}
\end{lem}
\begin{proof}
If $G  \in H^1_+(\mathbb{T}; \mathbb{C}^{P \times d})$, $F \in  H^1_+(\mathbb{T}; \mathbb{C}^{d \times Q})$ for some $d \in \mathbb{N}_+$, formula \eqref{Pi<0adjTorus} yields that
\begin{equation*}
\begin{split}
& \mathbf{H}^{(\mathbf{r})}_{U} \mathbf{H}^{(\mathbf{l})}_{V} (G) = \Pi_{\geq 0} \left(UV^* G- U \Pi_{\geq 0}(V^* G) \right) + U\widehat{V^* G}(0) = (\mathbf{T}^{(\mathbf{r})}_{U V^*}  - \mathbf{T}^{(\mathbf{r})}_{U}\mathbf{T}^{(\mathbf{r})}_{V^*}) (G) + U\widehat{V^* G}(0); \\
& \mathbf{H}^{(\mathbf{l})}_{V} \mathbf{H}^{(\mathbf{r})}_{W}(F) =  \Pi_{\geq 0} \left(F W^* V - \Pi_{\geq 0}(F W^*)V \right) +      \widehat{F W^*}(0)V =(\mathbf{T}^{(\mathbf{l})}_{W^* V}  - \mathbf{T}^{(\mathbf{l})}_{V}\mathbf{T}^{(\mathbf{l})}_{W^*})(F) + \widehat{F W^*}(0)V.
\end{split}
\end{equation*}It suffices to conclude by \eqref{rel2K2Hfor}.
\end{proof}

\begin{lem}
Given $P \in \mathbb{C}^{M \times d}$ and $Q \in \mathbb{C}^{d\times N}$ for some $M,N,d \in \mathbb{N}_+$, if $U\in H^{\frac{1}{2}}_+(\mathbb{T};  \mathbb{C}^{M\times N})$, then
\begin{equation}\label{HHP}
\mathbf{H}^{(\mathbf{r})}_{U} \mathbf{H}^{(\mathbf{l})}_{U} (P) =  \mathbf{T}^{(\mathbf{r})}_{U U^*}(P) \in  L^2_+(\mathbb{T};  \mathbb{C}^{M\times d}), \quad \mathbf{H}^{(\mathbf{l})}_{U} \mathbf{H}^{(\mathbf{r})}_{U} (Q) =  \mathbf{T}^{(\mathbf{l})}_{U^* U}(Q) \in  L^2_+(\mathbb{T};  \mathbb{C}^{d\times N}).
\end{equation}
\end{lem}
\begin{proof}
Trudinger's inequality \eqref{TruMatrixIne} yields that $U U^* \in L^2_+(\mathbb{T};  \mathbb{C}^{M\times M})$ and $U^* U \in L^2_+(\mathbb{T};  \mathbb{C}^{N\times N})$. Then \eqref{HHP} is obtained by \eqref{rlHankel}, \eqref{rlToep}  and \eqref{matrixXL^2}.
\end{proof}

\subsection{Proof of theorem $\ref{LaxPairThm}$} 
\begin{lem}\label{HUVWKUVWLem}
Given $s> \frac{1}{2}$ and $M,N,P,Q \in \mathbb{N}_+$, if $U \in H^s_+(\mathbb{T}; \mathbb{C}^{M \times N})$, $V \in H^s_+(\mathbb{T}; \mathbb{C}^{P \times N})$ and $W \in H^s_+(\mathbb{T}; \mathbb{C}^{P\times Q})$, then $\forall d \in \mathbb{N}_+$,  the following identities hold,
\begin{equation}\label{HUVWfor}
\begin{split}
\mathbf{H}^{(\mathbf{r})}_{\Pi_{\geq 0}(UV^*W)} = &\mathbf{T}^{(\mathbf{r})}_{UV^*}\mathbf{H}^{(\mathbf{r})}_W + \mathbf{H}^{(\mathbf{r})}_U \mathbf{T}^{(\mathbf{l})}_{W^* V} - \mathbf{H}^{(\mathbf{r})}_U \mathbf{H}^{(\mathbf{l})}_V \mathbf{H}^{(\mathbf{r})}_W: L^2_+(\mathbb{T}; \mathbb{C}^{d \times Q}) \to L^2_+(\mathbb{T}; \mathbb{C}^{M \times d});\\
\mathbf{H}^{(\mathbf{l})}_{\Pi_{\geq 0}(UV^*W)} = & \mathbf{T}^{(\mathbf{l})}_{V^* W}\mathbf{H}^{(\mathbf{l})}_U + \mathbf{H}^{(\mathbf{l})}_W \mathbf{T}^{(\mathbf{r})}_{V U^* } - \mathbf{H}^{(\mathbf{l})}_W \mathbf{H}^{(\mathbf{r})}_V \mathbf{H}^{(\mathbf{l})}_U: L^2_+(\mathbb{T}; \mathbb{C}^{M \times d}) \to L^2_+(\mathbb{T}; \mathbb{C}^{d \times Q}); \\
\mathbf{K}^{(\mathbf{r})}_{\Pi_{\geq 0}(UV^*W)} = &\mathbf{T}^{(\mathbf{r})}_{UV^*}\mathbf{K}^{(\mathbf{r})}_W + \mathbf{K}^{(\mathbf{r})}_U \mathbf{T}^{(\mathbf{l})}_{W^* V} - \mathbf{K}^{(\mathbf{r})}_U \mathbf{K}^{(\mathbf{l})}_V \mathbf{K}^{(\mathbf{r})}_W: L^2_+(\mathbb{T}; \mathbb{C}^{d \times Q}) \to L^2_+(\mathbb{T}; \mathbb{C}^{M \times d});\\
\mathbf{K}^{(\mathbf{l})}_{\Pi_{\geq 0}(UV^*W)} = & \mathbf{T}^{(\mathbf{l})}_{V^* W}\mathbf{K}^{(\mathbf{l})}_U + \mathbf{K}^{(\mathbf{l})}_W \mathbf{T}^{(\mathbf{r})}_{V U^* } - \mathbf{K}^{(\mathbf{l})}_W \mathbf{K}^{(\mathbf{r})}_V \mathbf{K}^{(\mathbf{l})}_U: L^2_+(\mathbb{T}; \mathbb{C}^{M \times d}) \to L^2_+(\mathbb{T}; \mathbb{C}^{d \times Q}). 
\end{split}
\end{equation}Equivalently,  $\forall d \in \mathbb{N}_+$, the following commutator formulas also hold:
\begin{equation}\label{commS*Toep(rlr)}
\begin{split}
\left[\mathbf{S}^*, \mathbf{T}^{(\mathbf{r})}_{UV^*}\right]\mathbf{H}^{(\mathbf{r})}_W  = &\mathbf{K}^{(\mathbf{r})}_U \left( \mathbf{H}^{(\mathbf{l})}_V \mathbf{H}^{(\mathbf{r})}_W  - \mathbf{K}^{(\mathbf{l})}_V \mathbf{K}^{(\mathbf{r})}_W\right): L^2_+(\mathbb{T}; \mathbb{C}^{d \times Q}) \to L^2_+(\mathbb{T}; \mathbb{C}^{M \times d}); \\
 \mathbf{H}^{(\mathbf{r})}_U \left[\mathbf{T}^{(\mathbf{l})}_{W^* V}, \mathbf{S}\right] = & \left(\mathbf{H}^{(\mathbf{r})}_U \mathbf{H}^{(\mathbf{l})}_V  -  \mathbf{K}^{(\mathbf{r})}_U \mathbf{K}^{(\mathbf{l})}_V \right) \mathbf{K}^{(\mathbf{r})}_W: L^2_+(\mathbb{T}; \mathbb{C}^{d \times Q}) \to L^2_+(\mathbb{T}; \mathbb{C}^{M \times d}); \\
\left[\mathbf{S}^*, \mathbf{T}^{(\mathbf{l})}_{V^*W }\right]\mathbf{H}^{(\mathbf{l})}_U  = &\mathbf{K}^{(\mathbf{l})}_W \left( \mathbf{H}^{(\mathbf{r})}_V \mathbf{H}^{(\mathbf{l})}_U  - \mathbf{K}^{(\mathbf{r})}_V \mathbf{K}^{(\mathbf{l})}_U\right): L^2_+(\mathbb{T}; \mathbb{C}^{M \times d}) \to L^2_+(\mathbb{T}; \mathbb{C}^{d \times Q}); \\
 \mathbf{H}^{(\mathbf{l})}_W \left[\mathbf{T}^{(\mathbf{r})}_{V U^*}, \mathbf{S}\right] = & \left(\mathbf{H}^{(\mathbf{l})}_W \mathbf{H}^{(\mathbf{r})}_V  -  \mathbf{K}^{(\mathbf{l})}_W \mathbf{K}^{(\mathbf{r})}_V \right) \mathbf{K}^{(\mathbf{l})}_U: L^2_+(\mathbb{T}; \mathbb{C}^{M \times d}) \to L^2_+(\mathbb{T}; \mathbb{C}^{d \times Q}). \\
\end{split}
\end{equation}
\end{lem}
\begin{proof}
If $F \in L^2_+(\mathbb{T}; \mathbb{C}^{d \times Q})$, since $UV^*W \in H^1(\mathbb{T}; \mathbb{C}^{M \times Q})$, we have $\Pi_{<0}(UV^* W)F^* \in L^2_-(\mathbb{T}; \mathbb{C}^{M \times d})$ by Lemma $\ref{AB*L2-}$. Formula \eqref{Pi<0adjTorus} yields that $\Pi_{<0}(WF^*) = \left(\Pi_{\geq 0}(F W^*) \right)^* -\widehat{WF^*}(0) \in L^2_-(\mathbb{T}; \mathbb{C}^{P \times d})$. Then
\begin{equation}\label{Hrpiuvw1}
\mathbf{H}^{(\mathbf{r})}_{\Pi_{\geq 0}(UV^*W)} (F) = \Pi_{\geq 0}(UV^* W F^*) =  \mathbf{T}^{(\mathbf{r})}_{UV^*}\mathbf{H}^{(\mathbf{r})}_W(F) + \mathbf{H}^{(\mathbf{r})}_{U} \left( \Pi_{\geq 0}(FW^*)V\right) -\Pi_{\geq 0}(UV^*)\widehat{WF^*}(0).
\end{equation}Using \eqref{Pi<0adjTorus} again, we obtain $\Pi_{\geq 0}(FW^*) = FW^* - \left(\Pi_{\geq 0}(WF^*)\right)^* + \widehat{F W^*}(0) \in L^2(\mathbb{T}; \mathbb{C}^{d \times P})$. Then 
\begin{equation}\label{PiFW*V}
\begin{split}
 \Pi_{\geq 0}(FW^*)V  =\Pi_{\geq 0} \left( \Pi_{\geq 0}(FW^*)V \right) = \mathbf{T}^{(\mathbf{l})}_{W^* V}(F) - \mathbf{H}^{(\mathbf{l})}_{V} \mathbf{H}^{(\mathbf{r})}_{W}(F) +   \widehat{F W^*}(0)V \in L^2(\mathbb{T}; \mathbb{C}^{d \times N}),
\end{split}
\end{equation}by using Lemma $\ref{ABL2+}$. Since $\widehat{F W^*}(0) = \left( \widehat{WF^*}(0) \right)^* \in \mathbb{C}^{d \times P}$, formula \eqref{matrixXL^2} implies that 
\begin{equation}\label{HUFW^*0V}
\mathbf{H}^{(\mathbf{r})}_{U} \left( \widehat{F W^*}(0)V\right) = \Pi_{\geq 0}(UV^*)\widehat{WF^*}(0) \in H^s_+(\mathbb{T}; \mathbb{C}^{M \times d}).
\end{equation}Plugging formulas \eqref{PiFW*V} and \eqref{HUFW^*0V} into  \eqref{Hrpiuvw1}, we obtain the first formula of \eqref{HUVWfor}.\\

\noindent If $G \in L^2_+(\mathbb{T}; \mathbb{C}^{M \times d})$, then $G^* U = \Pi_{\geq 0}(G^* U) + \left(\Pi_{\geq 0}(U^* G) \right)^* -\widehat{G^* U}(0) \in L^2 (\mathbb{T}; \mathbb{C}^{d \times N})$ by \eqref{Pi<0adjTorus}. Lemma $\ref{A*BL2-}$ implies that $G^* \Pi_{<0}(UV^* W) \in L^2_-(\mathbb{T}; \mathbb{C}^{d \times Q})$. As a consequence, we have
\begin{equation}\label{HlPiUV*WG}
\mathbf{H}^{(\mathbf{l})}_{\Pi_{\geq 0}(UV^*W)} (G) = \Pi_{\geq 0}( G^* UV^* W) =  \mathbf{T}^{(\mathbf{l})}_{V^* W}\mathbf{H}^{(\mathbf{l})}_U(G) + \mathbf{H}^{(\mathbf{l})}_{W} \left(V \Pi_{\geq 0}(U^*G) -V \widehat{U^* G}(0)\right).
\end{equation}because \eqref{matrixXL^2} yields $\mathbf{H}^{(\mathbf{l})}_{W} \left( V \widehat{U^* G}(0)\right) = \widehat{G^* U}(0)\Pi_{\geq 0}(V^*W)$. Lemma $\ref{ABL2+}$ and \eqref{Pi<0adjTorus} yield
\begin{equation}\label{VPiU*G}
V \Pi_{\geq 0}(U^*G) = \Pi_{\geq 0} \left(V \Pi_{\geq 0}(U^*G) \right) = \mathbf{T}^{(\mathbf{r})}_{V U^*}(G) - \mathbf{H}^{(\mathbf{r})}_{V}\mathbf{H}^{(\mathbf{l})}_{U}(G) + V \widehat{U^* G}(0) \in L^2(\mathbb{T}; \mathbb{C}^{P \times d}).
\end{equation}Plugging formula \eqref{VPiU*G} into \eqref{HlPiUV*WG}, we obtain the second formula of \eqref{HUVWfor}.\\

\noindent  Now we turn to prove the commutator formulas \eqref{commS*Toep(rlr)}. If $F\in L^2(\mathbb{T}; \mathbb{C}^{d \times Q})$, \eqref{matrixXL^2} and \eqref{ForCommTSS*} imply that $\mathbf{H}^{(\mathbf{r})}_{U} \left(\mathbf{e}_1 \widehat{FW^*}(0)V \right) = \Pi_{\geq 0}\left(\mathbf{e}_{-1} UV^* \right) \widehat{WF^*}(0)=\mathbf{S}^* (U V^*) \left(\mathbf{H}^{(\mathbf{r})}_{W}(F)\right)^{\wedge}(0)=\left[\mathbf{S}^* , \mathbf{T}^{(\mathbf{r})}_{U V^*}\right]\mathbf{H}^{(\mathbf{r})}_{W}(F)$. We have $\mathbf{K}^{(\mathbf{r})}_U \left( \mathbf{H}^{(\mathbf{l})}_V \mathbf{H}^{(\mathbf{r})}_W  - \mathbf{K}^{(\mathbf{l})}_V \mathbf{K}^{(\mathbf{r})}_W\right) (F) = \mathbf{H}^{(\mathbf{r})}_U \mathbf{S}\left( \widehat{FW^*}(0)V \right) = \left[\mathbf{S}^* , \mathbf{T}^{(\mathbf{r})}_{U V^*}\right]\mathbf{H}^{(\mathbf{r})}_{W}(F)$ by using \eqref{rel2K2Hfor}. If $G \in L^2_+(\mathbb{T}; \mathbb{C}^{M \times d})$, Lemma $\ref{AB*L2-}$ yields that $\Pi_{<0}\left(\mathbf{e}_{-1} G^*  U\right) V^* \in L^2_-(\mathbb{T}; \mathbb{C}^{d \times P})$, then
\begin{equation}\label{KulGV*0mode}
\Pi_{\geq 0}\left(\mathbf{K}^{(\mathbf{l})}_{U}(G) V^* \right)=\Pi_{\geq 0}\left(\mathbf{e}_{-1}G^* U V^* \right) \; \Rightarrow \;  ( \mathbf{K}^{(\mathbf{l})}_{U}(G) V^*)^{\wedge}(0) = (G^* U V^* )^{\wedge}(1).
\end{equation}Thanks to formula \eqref{ForCommTSS*}, \eqref{KulGV*0mode} and \eqref{rel2K2Hfor}, we have
\begin{equation*}
\begin{split}
\mathbf{H}^{(\mathbf{l})}_{W} \left[\mathbf{T}^{(\mathbf{r})}_{V U^*}, \mathbf{S} \right](G) = & \mathbf{H}^{(\mathbf{l})}_{W} (\left(V U^* G )^{\wedge}(-1) \right) = (G^* U V^* )^{\wedge}(1) W = ( \mathbf{K}^{(\mathbf{l})}_{U}(G) V^*)^{\wedge}(0) W \\
=& \left(\mathbf{H}^{(\mathbf{l})}_W \mathbf{H}^{(\mathbf{r})}_V  -  \mathbf{K}^{(\mathbf{l})}_W \mathbf{K}^{(\mathbf{r})}_V \right) \mathbf{K}^{(\mathbf{l})}_U (G).
\end{split}
\end{equation*}The first and the last formula of \eqref{commS*Toep(rlr)} are obtained. Together with the first two formulas of \eqref{HUVWfor}, we can deduce the last two formulas of \eqref{HUVWfor}. The second and the third formulas of \eqref{commS*Toep(rlr)} can be obtained by either comparing the first two formulas and the last two formulas of \eqref{HUVWfor} or following the same idea as the proof of the first and the last formula of \eqref{commS*Toep(rlr)} by using \eqref{ForCommTSS*} and \eqref{rel2K2Hfor}.
\end{proof}

\begin{proof}[Proof of theorem $\ref{LaxPairThm}$]
Given  $s>\frac{1}{2}$, set $V=W=U\in H^s_+(\mathbb{T};  \mathbb{C}^{M\times N})$ in formula \eqref{HUVWfor}. Then
\begin{equation}
\begin{split}
& \mathbf{H}^{(\mathbf{r})}_{U}\mathbf{H}^{(\mathbf{l})}_{\Pi_{\geq 0}(U U^* U)} - \mathbf{H}^{(\mathbf{r})}_{\Pi_{\geq 0}(U U^* U)} \mathbf{H}^{(\mathbf{l})}_{U} = \left[\mathbf{H}^{(\mathbf{r})}_{U}\mathbf{H}^{(\mathbf{l})}_{U}, \; \mathbf{T}^{(\mathbf{r})}_{U U^*}\right] \in \mathcal{B}(L^2_+(\mathbb{T};  \mathbb{C}^{M\times d})); \\
& \mathbf{H}^{(\mathbf{l})}_{U}\mathbf{H}^{(\mathbf{r})}_{\Pi_{\geq 0}(U U^* U)} - \mathbf{H}^{(\mathbf{l})}_{\Pi_{\geq 0}(U U^* U)} \mathbf{H}^{(\mathbf{r})}_{U} = \left[\mathbf{H}^{(\mathbf{l})}_{U}\mathbf{H}^{(\mathbf{r})}_{U}, \; \mathbf{T}^{(\mathbf{l})}_{U^* U}\right] \in \mathcal{B}(L^2_+(\mathbb{T};  \mathbb{C}^{d\times N})); \\
& \mathbf{K}^{(\mathbf{r})}_{U}\mathbf{K}^{(\mathbf{l})}_{\Pi_{\geq 0}(U U^* U)} - \mathbf{K}^{(\mathbf{r})}_{\Pi_{\geq 0}(U U^* U)} \mathbf{K}^{(\mathbf{l})}_{U} = \left[\mathbf{K}^{(\mathbf{r})}_{U}\mathbf{K}^{(\mathbf{l})}_{U}, \; \mathbf{T}^{(\mathbf{r})}_{U U^*}\right] \in \mathcal{B}(L^2_+(\mathbb{T};  \mathbb{C}^{M\times d})); \\
& \mathbf{K}^{(\mathbf{l})}_{U}\mathbf{K}^{(\mathbf{r})}_{\Pi_{\geq 0}(U U^* U)} - \mathbf{K}^{(\mathbf{l})}_{\Pi_{\geq 0}(U U^* U)} \mathbf{K}^{(\mathbf{r})}_{U} = \left[\mathbf{K}^{(\mathbf{l})}_{U}\mathbf{K}^{(\mathbf{r})}_{U}, \; \mathbf{T}^{(\mathbf{l})}_{U^* U}\right] \in \mathcal{B}(L^2_+(\mathbb{T};  \mathbb{C}^{d\times N})).
\end{split}
\end{equation}We conclude by the $\mathbb{C}$-antilinearity of the Hankel operators defined in \eqref{rlHankel} and \eqref{rlShiftHankel}. 
\end{proof}

\begin{rem}
Thanks to formula \eqref{rel2K&T-TTfor}, $(\mathbf{K}^{(\mathbf{r})}_{U}\mathbf{K}^{(\mathbf{l})}_{U}, -i \mathbf{T}^{(\mathbf{r})}_{U} \mathbf{T}^{(\mathbf{r})}_{ U^*} )$ and 
 $(\mathbf{K}^{(\mathbf{l})}_{U}\mathbf{K}^{(\mathbf{r})}_{U}, -i \mathbf{T}^{(\mathbf{l})}_{U} \mathbf{T}^{(\mathbf{l})}_{U^*})$ are also Lax pairs of the matrix Szeg\H{o} equation \eqref{MSzego}. 
\end{rem}

\section{The explicit formula}\label{SecExplicit}
\noindent This section is dedicated to establish the explicit formula of solutions to \eqref{MSzego}. Thanks to $Theorem \ref{LaxPairThm}$, the matrix Szeg\H{o} equation \eqref{MSzego} has at least $4$ Lax pairs: $(\mathbf{H}^{(\mathbf{r})}_{U}\mathbf{H}^{(\mathbf{l})}_{U}, -i \mathbf{T}^{(\mathbf{r})}_{U U^*} )$, $(\mathbf{H}^{(\mathbf{l})}_{U}\mathbf{H}^{(\mathbf{r})}_{U}, -i \mathbf{T}^{(\mathbf{l})}_{U^* U} )$, $(\mathbf{K}^{(\mathbf{r})}_{U}\mathbf{K}^{(\mathbf{l})}_{U}, -i \mathbf{T}^{(\mathbf{r})}_{U U^*} )$,
 $(\mathbf{K}^{(\mathbf{l})}_{U}\mathbf{K}^{(\mathbf{r})}_{U}, -i \mathbf{T}^{(\mathbf{l})}_{U^* U} )$. Then we have the following unitary equivalence corollary.

\begin{cor}\label{corUniEquivSzego}
Given $M,N,d \in \mathbb{N}_+$ and $s>\tfrac{1}{2}$, if $U \in C^{\infty}(\mathbb{R}; H^s_+(\mathbb{T};  \mathbb{C}^{M\times N}))$ solves equation \eqref{MSzego}, let $\mathbf{W}\in C^1(\mathbb{R}; \mathcal{B}(L^2_+(\mathbb{T};\mathbb{C}^{M \times d})))$ and $\mathscr{W}\in C^1(\mathbb{R}; \mathcal{B}(L^2_+(\mathbb{T};\mathbb{C}^{d \times N})))$ denote the unique solution to the following equation:
\begin{equation}\label{WODE}
\begin{split}
\tfrac{\mathrm{d}}{\mathrm{d}t}\mathbf{W}(t) = -i \mathbf{T}^{(\mathbf{r})}_{U(t) U(t)^*}\mathbf{W}(t), \quad \tfrac{\mathrm{d}}{\mathrm{d}t}\mathscr{W}(t) = -i \mathbf{T}^{(\mathbf{l})}_{U(t)^* U(t)}\mathscr{W}(t) 
\end{split}
\end{equation}with initial data $\mathbf{W}(0)= \mathrm{id}_{L^2_+(\mathbb{T};\mathbb{C}^{M \times d}))}$ and $\mathscr{W}(0)= \mathrm{id}_{L^2_+(\mathbb{T};\mathbb{C}^{d \times N}))}$. Then, for any $t \in \mathbb{R}$, $\mathbf{W}(t)$ and $\mathscr{W}(t)$ are both unitary operators and the following identities of unitary equivalences hold:
\begin{equation}\label{UniEquiHHKK}
\begin{split}
& \mathbf{H}^{(\mathbf{r})}_{U(t)}\mathbf{H}^{(\mathbf{l})}_{U(t)} = \mathbf{W}(t)\mathbf{H}^{(\mathbf{r})}_{U(0)}\mathbf{H}^{(\mathbf{l})}_{U(0)} \mathbf{W}(t)^*; \quad \mathbf{H}^{(\mathbf{l})}_{U(t)}\mathbf{H}^{(\mathbf{r})}_{U(t)} = \mathscr{W}(t)\mathbf{H}^{(\mathbf{l})}_{U(0)}\mathbf{H}^{(\mathbf{r})}_{U(0)} \mathscr{W}(t)^*;\\
& \mathbf{K}^{(\mathbf{r})}_{U(t)}\mathbf{K}^{(\mathbf{l})}_{U(t)} = \mathbf{W}(t)\mathbf{K}^{(\mathbf{r})}_{U(0)}\mathbf{K}^{(\mathbf{l})}_{U(0)} \mathbf{W}(t)^*;  \quad \mathbf{K}^{(\mathbf{l})}_{U(t)}\mathbf{K}^{(\mathbf{r})}_{U(t)} = \mathscr{W}(t)\mathbf{K}^{(\mathbf{l})}_{U(0)}\mathbf{K}^{(\mathbf{r})}_{U(0)} \mathscr{W}(t)^*.
\end{split}
\end{equation}
\end{cor}
\begin{proof}
Let $\mathbb{X}_{MN}:= \mathcal{B}(L^2_+(\mathbb{T}; \mathbb{C}^{M \times N}))$, $\forall M,N \in \mathbb{N}_+$. Both $\mathcal{A}^{(\mathbf{r})}: t \in \mathbb{R} \mapsto \mathcal{A}^{(\mathbf{r})}(t)   \in \mathcal{B}(\mathbb{X}_{Md})$ and $\mathcal{A}^{(\mathbf{l})}: t \in \mathbb{R} \mapsto \mathcal{A}^{(\mathbf{l})}(t)  \in   \mathcal{B}(\mathbb{X}_{dN})$ are continuous, where  $\mathcal{A}^{(\mathbf{r})}(t): \mathbf{W}\in \mathbb{X}_{Md} \mapsto -i \mathbf{T}^{(\mathbf{r})}_{U(t) U(t)^*}\mathbf{W} \in \mathbb{X}_{Md}$ and $\mathcal{A}^{(\mathbf{l})}(t): \mathscr{W}\in \mathbb{X}_{dN} \mapsto -i \mathbf{T}^{(\mathbf{l})}_{ U(t)^* U(t)}\mathscr{W} \in \mathbb{X}_{dN}$. Then \eqref{WODE} admits a unique solution thanks to Proposition $\ref{CauchyThmODE}$. Since both $\mathbf{T}^{(\mathbf{r})}_{U(t) U(t)^*} \in \mathbb{X}_{Md}$ and $\mathbf{T}^{(\mathbf{l})}_{ U(t)^* U(t)}  \in \mathbb{X}_{dN}$ are self-adjoint operators, we have $\mathbf{W}(t)^* = \mathbf{W}(t)^{-1} \in \mathbb{X}_{Md}$ and $\mathscr{W}(t)^* = \mathscr{W}(t)^{-1} \in \mathbb{X}_{dN}$ by uniqueness argument in Proposition $\ref{CauchyThmODE}$. Then  \eqref{4HeiLaxMSzego} yields that $\tfrac{\mathrm{d}}{\mathrm{d}t}(\mathbf{W}(t)^*  \mathbf{H}^{(\mathbf{r})}_{U(t)}\mathbf{H}^{(\mathbf{l})}_{U(t)}\mathbf{W}(t))=\tfrac{\mathrm{d}}{\mathrm{d}t}(\mathbf{W}(t)^*  \mathbf{K}^{(\mathbf{r})}_{U(t)}\mathbf{K}^{(\mathbf{l})}_{U(t)}\mathbf{W}(t))= 0_{\mathbb{X}_{Md}}$ and $\tfrac{\mathrm{d}}{\mathrm{d}t}(\mathscr{W}(t)^*  \mathbf{H}^{(\mathbf{l})}_{U(t)}\mathbf{H}^{(\mathbf{r})}_{U(t)}\mathscr{W}(t)) =\tfrac{\mathrm{d}}{\mathrm{d}t}(\mathscr{W}(t)^*  \mathbf{K}^{(\mathbf{l})}_{U(t)}\mathbf{K}^{(\mathbf{r})}_{U(t)}\mathscr{W}(t)) = 0_{\mathbb{X}_{dN}}$.
\end{proof}

\noindent The following lemma gives the relation of the family of unitary operators $(\mathbf{W}(t))_{t \in \mathbb{R}}$ and the unitary groups $(e^{it \mathbf{K}^{(\mathbf{r})}_{U(0)}\mathbf{K}^{(\mathbf{l})}_{U(0)}})_{t \in \mathbb{R}}$ and $(e^{it \mathbf{H}^{(\mathbf{r})}_{U(0)}\mathbf{H}^{(\mathbf{l})}_{U(0)}})_{t \in \mathbb{R}}$, which allows to linearize the matrix Szeg\H{o} flow. 
\begin{lem}
Given $M,N,d \in \mathbb{N}_+$ and $s>\tfrac{1}{2}$, if $U \in C^{\infty}(\mathbb{R}; H^s_+(\mathbb{T};  \mathbb{C}^{M\times N}))$ solves equation \eqref{MSzego},  $\mathbf{W}\in C^1(\mathbb{R}; \mathcal{B}(L^2_+(\mathbb{T};\mathbb{C}^{M \times d})))$ and $\mathscr{W}\in C^1(\mathbb{R}; \mathcal{B}(L^2_+(\mathbb{T};\mathbb{C}^{d \times N})))$ are defined by \eqref{WODE} of Corollary $\ref{corUniEquivSzego}$. Then the following identities hold,  $\forall t \in \mathbb{R}$:
\begin{equation}\label{W*S*WHH}
\begin{split}
& \mathbf{W}(t)^* \mathbf{S}^*\mathbf{W}(t) \mathbf{H}^{(\mathbf{r})}_{U(0)}\mathbf{H}^{(\mathbf{l})}_{U(0)} = e^{it \mathbf{K}^{(\mathbf{r})}_{U(0)}\mathbf{K}^{(\mathbf{l})}_{U(0)}} \mathbf{S}^* e^{-it \mathbf{H}^{(\mathbf{r})}_{U(0)}\mathbf{H}^{(\mathbf{l})}_{U(0)}} \mathbf{H}^{(\mathbf{r})}_{U(0)}\mathbf{H}^{(\mathbf{l})}_{U(0)} \in \mathcal{B}(L^2_+(\mathbb{T};\mathbb{C}^{M \times d})); \\
& \mathscr{W}(t)^* \mathbf{S}^*\mathscr{W}(t) \mathbf{H}^{(\mathbf{l})}_{U(0)}\mathbf{H}^{(\mathbf{r})}_{U(0)} = e^{it \mathbf{K}^{(\mathbf{l})}_{U(0)}\mathbf{K}^{(\mathbf{r})}_{U(0)}} \mathbf{S}^* e^{-it \mathbf{H}^{(\mathbf{l})}_{U(0)}\mathbf{H}^{(\mathbf{r})}_{U(0)}} \mathbf{H}^{(\mathbf{l})}_{U(0)}\mathbf{H}^{(\mathbf{r})}_{U(0)} \in \mathcal{B}(L^2_+(\mathbb{T};\mathbb{C}^{d \times N})).
\end{split}
\end{equation}Moreover, for any constant matrices $P \in \mathbb{C}^{M \times d}$ and $Q \in \mathbb{C}^{d\times N}$, we have
\begin{equation}\label{W*PQSze}
\mathbf{W}(t)^*(P) = e^{it \mathbf{H}^{(\mathbf{r})}_{U(0)}\mathbf{H}^{(\mathbf{l})}_{U(0)}}(P) \in L^2_+(\mathbb{T}; \mathbb{C}^{M \times d}); \quad \mathscr{W}(t)^*(Q) = e^{it \mathbf{H}^{(\mathbf{l})}_{U(0)}\mathbf{H}^{(\mathbf{r})}_{U(0)}}(Q) \in L^2_+(\mathbb{T}; \mathbb{C}^{d\times N}).
\end{equation}We also have 
\begin{equation}\label{W*U}
\mathbf{W}(t)^* (U(t)) =\mathscr{W}(t)^* (U(t)) = U(0) \in H^s_+(\mathbb{T};  \mathbb{C}^{M\times N}).
\end{equation}
\end{lem}
\begin{proof}
Set $\mathfrak{Y}(t):= \mathbf{W}(t)^* \mathbf{S}^*\mathbf{W}(t) \mathbf{H}^{(\mathbf{r})}_{U(0)}\mathbf{H}^{(\mathbf{l})}_{U(0)}$ and $\mathscr{Y}(t):=\mathscr{W}(t)^* \mathbf{S}^*\mathscr{W}(t) \mathbf{H}^{(\mathbf{l})}_{U(0)}\mathbf{H}^{(\mathbf{r})}_{U(0)}$, $\forall t \in \mathbb{R}$. Then formulas \eqref{commS*Toep(rlr)} and \eqref{UniEquiHHKK} yield that
\begin{equation}\label{Y'Yt'Sze}
\begin{split}
&\tfrac{\mathrm{d}}{\mathrm{d}t}\mathfrak{Y}(t)= -i \mathbf{W}(t)^* [\mathbf{S}^*, \mathbf{T}^{(\mathbf{r})}_{U(t) U(t)^*}]\mathbf{H}^{(\mathbf{r})}_{U(t)}\mathbf{H}^{(\mathbf{l})}_{U(t)} \mathbf{W}(t) \in \mathcal{B}(L^2_+(\mathbb{T};\mathbb{C}^{M \times d}))\\
 = & i \mathbf{W}(t)^* \mathbf{K}^{(\mathbf{r})}_{U(t)}\left(\mathbf{K}^{(\mathbf{l})}_{U(t)} \mathbf{K}^{(\mathbf{r})}_{U(t)} - \mathbf{H}^{(\mathbf{l})}_{U(t)} \mathbf{H}^{(\mathbf{r})}_{U(t)}\right)\mathbf{H}^{(\mathbf{l})}_{U(t)} \mathbf{W}(t) = i \mathbf{K}^{(\mathbf{r})}_{U(0)} \mathbf{K}^{(\mathbf{l})}_{U(0)} \mathfrak{Y}(t) -i \mathfrak{Y}(t) \mathbf{H}^{(\mathbf{r})}_{U(0)} \mathbf{H}^{(\mathbf{l})}_{U(0)};\\
&\tfrac{\mathrm{d}}{\mathrm{d}t}\mathscr{Y}(t)= -i \mathscr{W}(t)^* [\mathbf{S}^*, \mathbf{T}^{(\mathbf{l})}_{U(t)^* U(t)}]\mathbf{H}^{(\mathbf{l})}_{U(t)}\mathbf{H}^{(\mathbf{r})}_{U(t)} \mathscr{W}(t) \in \mathcal{B}(L^2_+(\mathbb{T};\mathbb{C}^{d \times N}))\\
 = & i \mathscr{W}(t)^* \mathbf{K}^{(\mathbf{l})}_{U(t)}\left(\mathbf{K}^{(\mathbf{r})}_{U(t)} \mathbf{K}^{(\mathbf{l})}_{U(t)} - \mathbf{H}^{(\mathbf{r})}_{U(t)} \mathbf{H}^{(\mathbf{l})}_{U(t)}\right)\mathbf{H}^{(\mathbf{r})}_{U(t)} \mathscr{W}(t) = i \mathbf{K}^{(\mathbf{l})}_{U(0)} \mathbf{K}^{(\mathbf{r})}_{U(0)} \mathscr{Y}(t) -i \mathscr{Y}(t) \mathbf{H}^{(\mathbf{l})}_{U(0)} \mathbf{H}^{(\mathbf{r})}_{U(0)}.\\
\end{split}
\end{equation}Then \eqref{W*S*WHH} is obtained by integrating \eqref{Y'Yt'Sze} and \eqref{HExpHH}. Formula \eqref{W*U} is obtained by \eqref{WODE} and the following expression of the matrix Szeg\H{o} equation \eqref{MSzego}:
\begin{equation}\label{U'(t)=A((U(t)))}
\partial_t U(t) = -i \mathbf{T}^{(\mathbf{r})}_{U(t) U(t)^*}(U(t)) = -i \mathbf{T}^{(\mathbf{l})}_{U(t)^* U(t)}(U(t)) \in H^s_+(\mathbb{T};  \mathbb{C}^{M\times N}).
\end{equation}If $P \in \mathbb{C}^{M \times d}$ and $Q \in \mathbb{C}^{d\times N}$, then $\partial_t (\mathbf{W}(t)^* (P))=i \mathbf{W}(t)^* \mathbf{H}^{(\mathbf{r})}_{U(t)} \mathbf{H}^{(\mathbf{l})}_{U(t)} (P) = i  \mathbf{H}^{(\mathbf{r})}_{U(0)} \mathbf{H}^{(\mathbf{l})}_{U(0)} \mathbf{W}(t)^* (P)$ and $\partial_t (\mathscr{W}(t)^* (Q))=i \mathscr{W}(t)^* \mathbf{H}^{(\mathbf{l})}_{U(t)} \mathbf{H}^{(\mathbf{r})}_{U(t)} (Q) = i  \mathbf{H}^{(\mathbf{l})}_{U(0)} \mathbf{H}^{(\mathbf{r})}_{U(0)} \mathscr{W}(t)^* (Q)$ by \eqref{WODE} and \eqref{HHP}.
\end{proof}

\noindent Finally, we act these three families of unitary operators on the shift operator $\mathbf{S}^*$ and complete the proof by conjugation acting method.
\begin{proof}[Proof of theorem $\ref{ExpForThm}$]
At first, assume that $U_0= U(0) \in \mathcal{M}_{\mathrm{FR}}^{M \times N}$ and $\mathbf{R}:=\dim_{\mathbb{C}} \mathrm{Im}\mathbf{H}^{(\mathbf{r})}_{U_0} \in \mathbb{N}$. Proposition $\ref{WKroneckerM*N}$ and the unitary equivalence property \eqref{UniEquiHHKK} yield that $U(t) \in \mathcal{M}_{\mathrm{FR}}^{M \times N}$ and   $\mathbb{V} :=\mathrm{Im}\mathbf{H}^{(\mathbf{r})}_{U_0} =  \mathrm{Im}\mathbf{H}^{(\mathbf{r})}_{U_0}\mathbf{H}^{(\mathbf{l})}_{U_0}$ is an $\mathbf{R}$-dimensional subspace of $L^2_+(\mathbb{T};  \mathbb{C}^{M\times N})$ such that
\begin{equation} 
\mathbf{W}(t)^* \mathbf{S}^*\mathbf{W}(t)\big|_{\mathbb{V}}   = e^{it \mathbf{K}^{(\mathbf{r})}_{U_0}\mathbf{K}^{(\mathbf{l})}_{U_0}} \mathbf{S}^* e^{-it \mathbf{H}^{(\mathbf{r})}_{U_0}\mathbf{H}^{(\mathbf{l})}_{U_0}}\big|_{\mathbb{V}}  : \mathbb{V} \to \mathbb{V} 
\end{equation}by \eqref{W*S*WHH}. Since $U_0 =\mathbf{H}^{(\mathbf{r})}_{U_0} (\mathbb{I}_N) \in \mathbb{V}$, thanks to the invariant-subspace-property \eqref{ImHuInv}, we have
\begin{equation}\label{1-zW*S*WMSze}
 (\mathrm{id}  - z \mathbf{W}(t)^* \mathbf{S}^*\mathbf{W}(t) )^{-1}(U_0) =  ( \mathrm{id}  - z e^{it \mathbf{K}^{(\mathbf{r})}_{U_0}\mathbf{K}^{(\mathbf{l})}_{U_0}} \mathbf{S}^* e^{-it \mathbf{H}^{(\mathbf{r})}_{U_0}\mathbf{H}^{(\mathbf{l})}_{U_0}})^{-1}(U_0) \in \mathbb{V},
\end{equation}$\forall z \in D(0,1)$. Then \eqref{1-zW*S*WMSze}, \eqref{W*PQSze} and \eqref{W*U} imply that
\begin{equation}\label{<>kjmsze}
\begin{split}
& \langle (\mathrm{id} -z \mathbf{S}^*)^{-1} U(t), \mathbb{E}^{(MN)}_{kj}\rangle_{L^2_+} = \langle (\mathrm{id} -z \mathbf{W}(t)^* \mathbf{S}^* \mathbf{W}(t))^{-1} \mathbf{W}(t)^*U(t), \mathbf{W}(t)^* \mathbb{E}^{(MN)}_{kj}\rangle_{L^2_+} \\
= &  \langle (\mathrm{id}  - z e^{-it \mathbf{H}^{(\mathbf{r})}_{U_0}\mathbf{H}^{(\mathbf{l})}_{U_0}}e^{it \mathbf{K}^{(\mathbf{r})}_{U_0}\mathbf{K}^{(\mathbf{l})}_{U_0}} \mathbf{S}^* )^{-1}  e^{-it \mathbf{H}^{(\mathbf{r})}_{U_0}\mathbf{H}^{(\mathbf{l})}_{U_0}} (U_0) , \mathbb{E}^{(MN)}_{kj}\rangle_{L^2_+(\mathbb{T}; \mathbb{C}^{M \times N})}.
\end{split}
\end{equation}The Poisson integral of $U(t)=\sum_{n \geq 0}  \hat{U}(t, n) \mathbf{e}_n \in \mathcal{M}_{\mathrm{FR}}^{M \times N}$ is given by
\begin{equation}\label{PoiUMsze}
\underline{U}(t, z) = \sum_{n \geq 0} z^n \hat{U}(t, n) = \sum_{k=1}^M \sum_{j=1}^N \langle (\mathrm{id} -z \mathbf{S}^*)^{-1} U(t), \mathbb{E}^{(MN)}_{kj}\rangle_{L^2_+(\mathbb{T}; \mathbb{C}^{M \times N})}\mathbb{E}^{(MN)}_{kj} \in \mathbb{C}^{M \times N}.
\end{equation}thanks to  \eqref{I<>formula} and \eqref{InvForTorus}. Plugging formula \eqref{<>kjmsze} into \eqref{PoiUMsze}, we deduce that
\begin{equation}\label{UtzMszerl}
 \underline{U}(t, z)   = \mathbf{I} \left((\mathrm{id}  - z e^{-it \mathbf{H}^{(\mathbf{r})}_{U_0}\mathbf{H}^{(\mathbf{l})}_{U_0}}e^{it \mathbf{K}^{(\mathbf{r})}_{U_0}\mathbf{K}^{(\mathbf{l})}_{U_0}} \mathbf{S}^* )^{-1}  e^{-it \mathbf{H}^{(\mathbf{r})}_{U_0}\mathbf{H}^{(\mathbf{l})}_{U_0}} (U_0)  \right).
\end{equation}by \eqref{I<>formula} again. Similarly, since $U_0 =\mathbf{H}^{(\mathbf{l})}_{U_0} (\mathbb{I}_M) \in \mathscr{V}:=\mathrm{Im}\mathbf{H}^{(\mathbf{l})}_{U_0} =  \mathrm{Im}\mathbf{H}^{(\mathbf{l})}_{U_0}\mathbf{H}^{(\mathbf{r})}_{U_0}$, which is an $\mathbf{R}$-dimensional subspace of $L^2_+(\mathbb{T};  \mathbb{C}^{M\times N})$ such that $\mathscr{W}(t)^* \mathbf{S}^*\mathscr{W}(t)\big|_{\mathscr{V}}    = e^{it \mathbf{K}^{(\mathbf{l})}_{U(0)}\mathbf{K}^{(\mathbf{r})}_{U(0)}} \mathbf{S}^* e^{-it \mathbf{H}^{(\mathbf{l})}_{U(0)}\mathbf{H}^{(\mathbf{r})}_{U(0)}}\big|_{\mathscr{V}}  :  \mathscr{V}  \to \mathscr{V}$ by \eqref{W*S*WHH}, then $(\mathrm{id}  - z \mathscr{W}(t)^* \mathbf{S}^*\mathscr{W}(t) )^{-1}(U_0) =  ( \mathrm{id}  - z e^{it \mathbf{K}^{(\mathbf{l})}_{U_0}\mathbf{K}^{(\mathbf{r})}_{U_0}} \mathbf{S}^* e^{-it \mathbf{H}^{(\mathbf{l})}_{U_0}\mathbf{H}^{(\mathbf{r})}_{U_0}})^{-1}(U_0) \in \mathscr{V}$. Following the previous steps, we  substitute $\mathscr{W}(t)$ for  $\mathbf{W}(t)$ in \eqref{<>kjmsze} and obtain that
\begin{equation}\label{UtzMszelr}
 \underline{U}(t, z)   = \mathbf{I} \left((\mathrm{id}  - z e^{-it \mathbf{H}^{(\mathbf{l})}_{U_0}\mathbf{H}^{(\mathbf{r})}_{U_0}}e^{it \mathbf{K}^{(\mathbf{l})}_{U_0}\mathbf{K}^{(\mathbf{r})}_{U_0}} \mathbf{S}^* )^{-1}  e^{-it \mathbf{H}^{(\mathbf{l})}_{U_0}\mathbf{H}^{(\mathbf{r})}_{U_0}} (U_0)  \right).
\end{equation}Expand $\underline{U}(t, z)$ in \eqref{UtzMszerl} and \eqref{UtzMszelr} into power series of $z \in D(0,1)$. Then \eqref{mszeExpFor} holds for $U_0  \in \mathcal{M}_{\mathrm{FR}}^{M \times N}$. \\

\noindent For general $U_0 \in H^{\frac{1}{2}}_+(\mathbb{T};  \mathbb{C}^{M\times N})$, it suffices to use the following approximation argument: the mappings $V \mapsto  (e^{-it  \mathbf{H}^{(\mathbf{r})}_{V }\mathbf{H}^{(\mathbf{l})}_{V }} e^{it  \mathbf{K}^{(\mathbf{r})}_{V }\mathbf{K}^{(\mathbf{l})}_{V }} \mathbf{S}^*)^n e^{-it  \mathbf{H}^{(\mathbf{r})}_{V }\mathbf{H}^{(\mathbf{l})}_{V }}  (V )$, $V \mapsto  (e^{-it  \mathbf{H}^{(\mathbf{l})}_{V }\mathbf{H}^{(\mathbf{r})}_{V }} e^{it  \mathbf{K}^{(\mathbf{l})}_{V }\mathbf{K}^{(\mathbf{r})}_{V }} \mathbf{S}^*)^n e^{-it  \mathbf{H}^{(\mathbf{l})}_{V }\mathbf{H}^{(\mathbf{r})}_{V }}  (V )$ and the flow map $U_0=U(0) \mapsto U(t)$ are all continuous from $H^{\frac{1}{2}}_+(\mathbb{T};  \mathbb{C}^{M\times N})$ to $L^2_+(\mathbb{T};  \mathbb{C}^{M\times N})$, $\forall (n,t) \in \mathbb{N} \times \mathbb{R}$, thanks to identity \eqref{TrHrHl} and Proposition $\ref{GWPH0.5}$. The proof is completed thanks to Lemma $\ref{denseFR}$, i.e. the density of $\mathcal{M}_{\mathrm{FR}}^{M \times N}$ in $H^{\frac{1}{2}}_+(\mathbb{T};  \mathbb{C}^{M\times N})$.
\end{proof} 
 
\begin{center}
$Acknowledgments$
\end{center}
On the one hand, the author would like to warmly thank Prof. Patrick G\'erard and Prof. Sandrine Grellier for all the lectures, presentations and mini-courses, which make the author very familiar to the cubic scalar Szeg\H{o} equation. On the other hand, the author is also grateful to Georgia Institute of Technology for both financially supporting the author's research and assigning the courses \textbf{Math1554} and \textbf{Math1553} about \textit{Linear Algebra} to the author, which makes the author more familiar to the matrix theory, leading to this matrix generalization of the cubic Szeg\H{o} equation.

\end{document}